\documentclass[10pt]{amsart}
\usepackage{amsmath,amsthm,amsfonts,amssymb,amscd,flafter,pinlabel}
\usepackage{mathtools}
\usepackage[mathscr]{eucal}
\usepackage{graphics}
\usepackage[all]{xy}
\usepackage{caption, subcaption}
\usepackage[usenames,dvipsnames]{color}
\usepackage{graphicx}

\usepackage{pgf}
\usepackage{tikz-cd}
\usetikzlibrary{arrows,automata}
\usepackage[latin1]{inputenc}

\usepackage[colorlinks=true, urlcolor=NavyBlue, linkcolor=NavyBlue, citecolor=NavyBlue]{hyperref}

\newtheorem{theorem}{Theorem}[section]

\newtheorem{question}[theorem]{Question}
\newtheorem{proposition}[theorem]{Proposition}

\newtheorem{lemma}[theorem]{Lemma}

\makeatletter
\newtheorem*{rep@theorem}{\rep@title}
\newcommand{\newreptheorem}[2]{%
	\newenvironment{rep#1}[1]{%
		\def\rep@title{#2 \ref{##1}}%
		\begin{rep@theorem}}%
		{\end{rep@theorem}}}
\makeatother

\newreptheorem{theorem}{Theorem}
\newreptheorem{lemma}{Lemma}
\newreptheorem{question}{Question}
\newreptheorem{corollary}{Corollary}
\newreptheorem{proposition}{Proposition}

\theoremstyle{definition}
\newtheorem{definition}[theorem]{Definition}

\theoremstyle{definition}
\newtheorem{remark}[theorem]{Remark}

\newtheorem{example}[theorem]{Example}
\newtheorem{examples}[theorem]{Examples}

\newcommand{\R}{\ensuremath{\mathbb{R}}}
\newcommand{\Z}{\ensuremath{\mathbb{Z}}}

\newcommand\TT{\mathbb{T}}

\newcommand{\cD}{\mathcal{D}}

\newcommand{\cT}{\mathcal{T}}
\newcommand{\cP}{\mathcal{P}}

\newcommand{\cC}{\mathcal{C}}

\newcommand{\cS}{\mathcal{S}}

\newcommand{\C}{\mathbb{C}}
\newcommand{\CP}{\mathbb{CP}}
\newcommand{\cptwo}{\CP^2}
\newcommand{\cpone}{\CP^1}

\DeclareRobustCommand\longtwoheadrightarrow
     {\relbar\joinrel\twoheadrightarrow}

\begin{document}
	
	\title{Symplectic 4--manifolds admit Weinstein trisections}
	
	\author[P. Lambert-Cole]{Peter Lambert-Cole}
	\address{Department of Mathematics \\ University of Georgia}
	\email{plc@uga.edu}
	
	\author[J. Meier]{Jeffrey Meier}
	\address{Department of Mathematics \\ Western Washington University}
	\email{jeffrey.meier@wwu.edu}
	\urladdr{\href{https://jeffreymeier.org}{https://jeffreymeier.org}}
	
	\author[L. Starkston]{Laura Starkston}
	\address{Department of Mathematics \\ University of California, Davis}
	\email{lstarkston@math.ucdavis.edu}
	\urladdr{\href{https://www.math.ucdavis.edu/~lstarkston/}{https://www.math.ucdavis.edu/\~{}lstarkston/}}
	
	\maketitle
	

	\begin{abstract}
		
		We prove that every symplectic 4--manifold admits a trisection that is compatible with the symplectic structure in the sense that the symplectic form induces a Weinstein structure on each of the three sectors of the trisection. Along the way, we show that a (potentially singular) symplectic braided surface in $\CP^2$ can be symplectically isotoped into bridge position.
		
	\end{abstract}

	\section{Introduction}
	\label{sec:Intro}
	
	A trisection of a $4$--manifold is a decomposition into three basic pieces with nice overlaps, which allows one to study the smooth $4$--manifold using curves on a surface up to isotopy and certain combinatorial equivalence moves. A trisection can be viewed as an alternative to a handle decomposition, the advantage being that curves on a surface are simpler than the framed links which serve as attaching spheres of handles. Although trisections are designed to study smooth topology, we can ask in which ways a trisection can keep track of additional geometric structures on a $4$--manifold. In this article, the geometric structure of interest is a symplectic structure.
	
	In the case that $(X,\omega)$ is a symplectic 4--manifold, it is natural to look for trisections on $X$ that are compatible in some way with $\omega$.  The standard trisection of $\CP^2$ (see Example \ref{ex:CP2-standard}) can be constructed via toric geometry and therefore interacts nicely with the symplectic geometry. The contact type property of the boundaries of the three sectors in $\cptwo$ played an important role in the trisection-theoretic proof of the Thom conjecture~\cite{LC-Thom}. Further progress relating trisections and symplectic structures was made by Gay:  Every closed symplectic manifold admits a Lefschetz pencil, and Gay described how to construct a trisection from a pencil~\cite{Gay-Lefschetz}; see also related work on Lefschetz fibrations by Baykur-Saeki~\cite{BaySae} and Castro-Ozbagci~\cite{CasOzb}.
	
	Weinstein provided a method of building symplectic manifolds through handle attachments \cite{Weinstein}, but the construction has a significant limitation. The handles for a symplectic manifold of dimension $2n$ can have index at most $n$, and for symplectic $4$--manifolds, even the $2$--handles can only be attached with certain framings (less than the maximal Thurston-Bennequin framing of a Legendrian realization of the attaching circle; see~\cite[Chapter~11]{GS} and~\cite{Gompf_Stein} for details).  There is no general method to extend a symplectic structure over a $4$--dimensional $3$--handle or $4$--handle. In particular, any closed symplectic $4$--manifold cannot be built solely from Weinstein's handle construction.
	
	The main result of this article demonstrates a new advantage that trisections enjoy over handle decompositions, namely compatibility with a symplectic structure on a closed manifold. In~\cite{LM-Complex}, the first two authors proposed the notion of a \emph{Weinstein trisection} of $(X,\omega)$ as the appropriate generalization of the standard trisection of $\CP^2$, and conjectured that every symplectic 4--manifold admits a Weinstein trisection.  In this paper, we prove that conjecture.
	
	\begin{theorem}
		\label{thm:Main}
		Every closed, symplectic 4--manifold $(X,\omega)$ admits a Weinstein trisection.
	\end{theorem}
		
	The proof of Theorem~\ref{thm:Main} relies on important results of Auroux \cite{Auroux} and Auroux-Katzarkov \cite{Auroux-Katzarkov} that every closed, symplectic 4--manifold admits a quasiholomorphic branched covering map $f\colon X \rightarrow \CP^2$.  Trisections are naturally suited for studying branched coverings of 4--manifolds (e.g. \cite{BCKM,RT-Multi,LM-Complex,CK}), and the basic strategy is to pull back the Weinstein trisection of $\CP^2$ via the singular branched covering map $f$.
	
	To carry this strategy out, we must first symplectically isotope the branch locus, which is possibly singular, into {\it  bridge position} with respect to a genus one Weinstein trisection of $\CP^2$.  Auroux and Katzarkov proved that the branch locus can be assumed to be braided with respect to the standard linear pencil~\cite{Auroux-Katzarkov}.  We show in Section~\ref{sec:bridge} that all such curves can be symplectically isotoped into bridge position.
	
	Putting this result (Theorem~\ref{thm:symplbridgepos}) together with the fact that any smooth symplectic submanifold (or with singularities modeled on a complex algebraic curve) in $\cptwo$ can be assumed to be braided (see Remark \ref{smoothisbraided}), we obtain the following theorem, which may shed some light on the study of the symplectic isotopy problem using bridge trisections.
		
	\begin{theorem}
		\label{thm:symp_isot}
		Every smooth (or algebraically singular) symplectic surface in $\cptwo$ is symplectically isotopic to one in bridge position with respect to a genus one Weinstein trisection of $\cptwo$.
	\end{theorem}
		
	Once we know that we can symplectically isotope our branch locus into bridge position, to prove that every symplectic manifold has a Weinstein trisection, it remains to show that (a) pulling back the (topological) trisection of $\CP^2$ by $f$ gives a topological trisection of $X$, and (b) pulling back the Weinstein structure on each sector by $f$ gives a Weinstein structure compatible with the global symplectic form.  These tasks are carried out in Sections~\ref{sec:sing} and~\ref{sec:lifting}, respectively. 
	We conclude with a discussion of examples in Section~\ref{sec:Examples}.

	\subsection*{Acknowledgements}
	This project arose thanks to the support of an AIM SQuaRE grant.  The authors would like to express their gratitude to the American Institute of Mathematics for providing an ideal collaborative environment, as well as acknowledge Paul Melvin, Juanita Pinz\'on-Caicedo, and Alex Zupan for their input at the beginning of the project. JM was supported by NSF grant DMS-1933019. LS was supported by NSF grant DMS-1904074. The authors also thank the Max Planck Institute for Mathematics for its hospitality. Finally, the authors would like to thank the anonymous referee for thorough and important suggestions and comments.
	
	\section{Definitions, notions of equivalence, and open questions}
	\label{sec:ques}

In this section, we give formal definitions for the concepts from the introduction, discuss ways in which two Weinstein trisections could be considered equivalent, and present a handful of open questions that we hope will direct and encourage future development.

	A {\it trisection} $\TT$ of a 4--manifold $X$ is a decomposition $X = Z_1 \cup Z_2 \cup Z_3$ into three pieces such that, for each $\lambda\in\Z_3$,
	\begin{enumerate}
		\item Each $Z_{\lambda}$ is diffeomorphic to $\natural^{k_{\lambda}}(S^1 \times B^3)$, for some $k_{\lambda}$,
		\item Each double intersection $H_{\lambda} = Z_{\lambda-1} \cap Z_{\lambda}$ is a 3--dimensional, genus $g$ handlebody, and
		\item The triple intersection $\Sigma = Z_1 \cap Z_2 \cap Z_3$ is a closed, genus $g$ surface, with $\Sigma = \partial H_\lambda$.
	\end{enumerate}
	The surface $\Sigma$ is called the \emph{core} of the trisection, and the genus of the core is the \emph{genus} of the trisection.
	
	Gay and Kirby showed that every closed, oriented, smooth 4--manifold $X$ admits a trisection $\TT$ and that any two trisections $\TT$ and $\TT'$ of $X$ have a common stabilization~\cite{GK}.  Trisections of low genus have been classified~\cite{MZ-Genus,MSZ}, and trisections for many familiar 4--manifolds, including complex hyper-surfaces in $\CP^3$, have been described~\cite{LM-Complex}.
	
	Suppose that $X$ admits a symplectic structure $\omega$ and a trisection $\TT$.  We want to understand whether the symplectic form is compatible with the $1$--handlebody structure on each of the three sectors $Z_\lambda$. Weinstein's handle construction gives a notion of compatibility between a symplectic structure and a $1$--handlebody structure~\cite{Weinstein}.
	
	We now give a more analytic way to keep track of Weinstein's handle structure. A \emph{Weinstein structure} on a 4--manifold $W$ is a quadruple $(W,\omega, V, \phi)$ where $W$ is a smooth manifold with boundary, $\omega$ is the symplectic structure, $V$ is a Liouville vector field for $\omega$ which is outwardly transverse to the boundary, and $\phi\colon W\to\R$ is a Morse function such that $V$ is gradient-like for $\phi$. To say a vector field $V$ is Liouville for $\omega$ means that the $1$--form $\eta:=\iota_V\omega$ obtained by taking the interior product of $\omega$ by $V$ (i.e. $\eta(\cdot) = \omega(V,\cdot)$) satisfies $d\eta=\omega$. To say that $V$ is gradient-like for $\phi$ implies that the zeros of $V$ are the critical points of $\phi$ and $d\phi(V)>0$ whenever $V\neq 0$. The data of the function $\phi$ is less important than the Liouville vector field $V$ (because $V$ is what must be compatible with the symplectic form and different choices of $\phi$ for a fixed $V$ are equivalent), but the \emph{existence} of some function $\phi$ for which $V$ is gradient-like is important. The existence of $\phi$ is equivalent to asking that the zeros of $V$ are locally modeled on the gradient of a Morse function, and $V$ should have no limit cycles (no periodic orbits and no oriented cyclic loops of trajectories) \cite{Sullivan}. Therefore, we can focus solely on the Liouville vector field $V$, as long as we control its trajectories and zeros. A key feature of a Weinstein structure is that it induces a contact structure on the boundary of the handlebody. The contact form is given by the restriction of $\eta$ to the boundary.
	
	Each piece $Z_{\lambda}$ of the trisection is a 4--dimensional 1--handlebody and therefore admits a (sub-critical) Weinstein structure $(Z_{\lambda},\omega_{\lambda},V_{\lambda},\phi_{\lambda})$. Thus, one can ask whether such Weinstein structures on the $Z_\lambda$ can be chosen compatibly with the global symplectic form $\omega$ on $X$.
	
	\begin{definition}
		A {\it Weinstein trisection} of a symplectic 4--manifold $(X,\omega)$ consists of a trisection $\TT$, with induced decomposition $X = Z_1 \cup Z_2 \cup Z_3$, such that there exists a Weinstein structure $(Z_{\lambda},\omega|_{Z_{\lambda}},V_{\lambda},\phi_{\lambda})$ on each sector using the restriction of the symplectic form $\omega$.
		In this case, we say $\TT$ and $\omega$ are \emph{compatible}.
	\end{definition}
	
\begin{remark}
	Note that the sectors $Z_\lambda$ are manifolds with boundary \emph{and corners}, where the corner is along the trisection surface $\Sigma$. Typically, we define a Weinstein domain structure on a manifold with boundary, requiring the Liouville vector field to be outwardly transverse to the boundary. In this situation with corners, we take a sequence of approximating hypersurfaces approaching the boundary which smooth the corner (in the canonical manner), and require that the Liouville vector field is transverse to sufficiently nearby smoothings of the boundary.
\end{remark}
	
\begin{example}
\label{ex:CP2-standard}
	Consider $\cptwo$ with homogeneous coordinates $[z_1:z_2:z_3]$, and consider the subsets
	$$Z_\lambda = \{[z_1:z_2:z_3] \mid |z_\lambda|,|z_{\lambda+1}|\leq |z_{\lambda+2}| \}$$
	$$H_\lambda = \{[z_1:z_2:z_3] \mid |z_\lambda|\leq|z_{\lambda+1}|= |z_{\lambda+2}| \}.$$
	
	This decomposition yields a Weinstein trisection $\TT_0$ for $(\cptwo,\omega_{FS})$ first given in~\cite{LM-Complex}, where $\omega_{FS}$ denotes the Fubini-Study symplectic form. This trisection is compatible with the toric structure on $\cptwo$, and Figure~\ref{fig:toric} shows the decomposition in the image of the moment map.
	
	The corners of $\partial Z_\lambda$ can be smoothed by approximating $\partial Z_\lambda$ by $f_{\lambda,N}^{-1}(1)$ where
	$$f_{\lambda,N}(z_\lambda,z_{\lambda+1}):= \frac{1}{N}(|z_\lambda|^2+|z_{\lambda+1}|^2)+|z_\lambda|^{2N} +|z_{\lambda+1}|^{2N}$$
	for $N>>0$ (as in \cite{LC-Thom}).
\end{example}

\begin{figure}[h!]
	\centering
	\includegraphics[scale=.7]{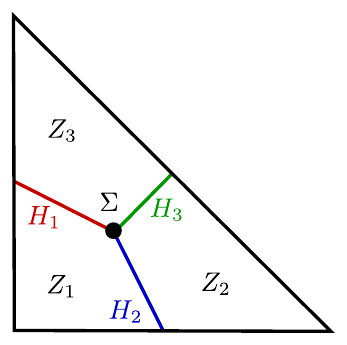}
	\caption{The standard genus one Weinstein trisection of $\cptwo$ coming from toric geometry.}
	\label{fig:toric}
\end{figure}

\begin{definition}
	Two Weinstein trisections $(X,\TT,\omega_0)$ and $(X,\TT,\omega_1)$ are \emph{isotopic} if there exists a family of Weinstein trisections $(X,\TT,\omega_t)$, with $t\in[0,1]$, such that $\omega_t$ is a family of symplectic forms interpolating between $\omega_0$ and $\omega_1$ with $[\omega_t]\in H^2(X;\R)$ independent of $t$.
\end{definition}

Given such a family of symplectic forms, Moser's trick implies that there exists an ambient isotopy $\Psi_t\colon X\to X$ such that $\Psi_0=id$ and $\Psi_t^*(\omega_t) = \omega_0$. Therefore, we could equivalently consider the family $(X,\TT_t, \omega_0)$ where the trisection $\TT_t = \Psi_t(\TT)$ varies through an isotopy and the symplectic form is fixed.

Even if two symplectic strutures are isotopic, and both are compatible with the same trisection, it is not clear whether the isotopy between them can remain compatible with the fixed trisection.
	
\begin{question}\label{ques:isot}
	Suppose that $(X,\TT,\omega_0)$ and $(X,\TT,\omega_1)$ are Weinstein trisections and that $\omega_0$ and $\omega_1$ are isotopic via a family $\omega_t$ such that $[\omega_t]\in H^2(X;\R)$ is independent of $t$.  Are $(X,\TT,\omega_0)$ and $(X,\TT,\omega_1)$ necessarily isotopic as Weinstein trisections?
\end{question}

Recall that the \emph{trisection genus} of a 4--manifold $X$ is the minimum value $g$ such that $X$ admits a trisection with core surface of genus $g$.

\begin{definition}
	The \emph{Weinstein trisection genus} of a symplectic 4--manifold $(X,\omega)$ is the minimum value $g$ such that $(X,\omega)$ admits a Weinstein trisection with core surface of genus $g$.
\end{definition}

\begin{question}\label{ques:genus}
	Is the Weinstein trisection genus of a symplectic 4--manifold $(X,\omega)$ always the same as the trisection genus of $X$?
\end{question}

Note that there are examples of symplectic 4--manifolds whose Weinstein trisection genus (for some symplectic form) equals their trisection genus.  These include complex hypersurfaces in $\CP^3$ and the elliptic surfaces $E(n)$; compare~\cite[Theorem~1.1]{LM-Complex} with Theorem~\ref{thm:pullbackWeiTri} and Example~\ref{ex:hyper} below.

Given a genus $g$ trisection $\TT$ for a 4--manifold $X$, there is a natural process, called \emph{stabilization}, that can be applied to produce a genus $g+1$ trisection $\TT'$ of $X$. Gay and Kirby proved that any two trisections of $X$ have a common stabilization. (See~\cite{MSZ} and~\cite{GK} for details.) If we start with a genus $g$ Weinstein trisection, we can perform this stabilization in a way that yields a genus $g+1$ Weinstein trisection.

\begin{proposition}
\label{prop:stab}
	Let $(X,\TT,\omega)$ be a genus $g$ Weinstein trisection, then any trisection stabilization can be performed such that the resulting $g+1$ trisection is compatible with $\omega$.
\end{proposition}

\begin{proof}
	The stabilization operation of a trisection can be understood in the following manner. Choose one sector $Z_\lambda$. Then $H_{\lambda-1}$ is the handlebody which is not contained in $\partial Z_\lambda$. Let $a$ be a properly-embedded, boundary-parallel arc in $H_{\lambda-1}$. Any stabilization of the trisection is obtained from some such choice of arc, by adding a regular neighborhood of $a$ to $Z_\lambda$ and deleting this regular neighborhood from $Z_{\lambda\pm 1}$. This has the effect of adding a $1$--handle to $Z_\lambda$. It does not change the topology of $Z_{\lambda\pm 1}$ because the arc $a$ lay in the boundary of each of these sectors, so deleting its neighborhood simply carves out a ``bite'' from their boundaries. The new pairwise intersections are still $3$-dimensional handlebodies when $a$ is boundary parallel.
	
	Now we consider how this construction interacts with the symplectic form. Any choice of arc $a$ is isotropic, so it has a standard symplectic neighborhood in $(X,\omega)$. This neighborhood can be identified with Weinstein's model for a symplectic $1$--handle attachment (identifying $a$ with the core of the $1$--handle). In particular, the Weinstein structure on $Z_\lambda$ extends over any sufficiently small neighborhood of $a$ compatibly with the fixed symplectic form $\omega$.

\begin{figure}[h!]
	\centering
	\includegraphics[scale=.6]{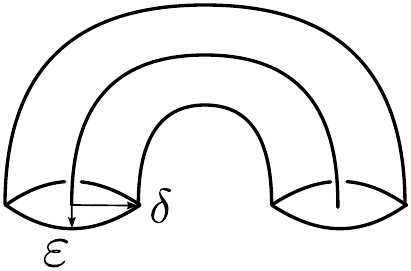}
	\caption{Dimensionally reduced neighborhood $N$ of the arc $a$ giving a Weinstein stabilization.}
	\label{fig:stab}
\end{figure}

	To ensure that the Weinstein structures on $Z_{\lambda\pm 1}$ remain Weinstein after deleting the neighborhood of $a$, we choose the neighborhood $N$ of $a$ so that the ``bite'' is sufficiently \emph{shallow}. We will simply restrict the original Weinstein structure on $Z_{\lambda\pm 1}$ to $Z_{\lambda\pm 1}\setminus N$. As long as the boundary $\partial (Z_{\lambda\pm 1}\setminus N)$ is sufficiently $C^1$--close to the boundary $\partial Z_{\lambda \pm 1}$, the Liouville vector field $V_{\lambda\pm 1}$ will still be outwardly transverse to the boundary. To choose $N$ to satisfy this, choose symplectic coordinates $(x_1,y_1,x_2,y_2)$ on a neighborhood of $a$ in $X$ such that $H_{\lambda-1} = \{y_2=0\}$ and $a=\{y_1=x_2=y_2=0\}$. Let
	$$N=\{(x_1,y_1,x_2,y_2) \mid y_1^2+x_2^2\leq h(|y_2|)^2, \; y_2\in [-\varepsilon,\varepsilon] \}$$
	where $h\colon [0,\varepsilon]\to [0,\delta]$ is a monotonically decreasing function. See Figure~\ref{fig:stab} for a dimensionally reduced (there is only one dimension representing both the $y_1$ and $x_2$ directions) picture for the shape of $N$. Choosing $\varepsilon/\delta$ sufficiently small ensures that $N$ is sufficiently shallow so the new boundary is $C^1$--close to the old boundary. Note that because of the absolute value in the argument for $h$ in the definition of $N$, $N$ has boundary and corners, but this is expected in a trisection, and we smooth corners as discussed above.
\end{proof}

Even if the answer to Question~\ref{ques:isot} is negative, it could be possible that after stabilization, the answer becomes positive.

\begin{question} \label{ques:stabisot}
	Suppose that $(X,\TT,\omega_0)$ and $(X,\TT,\omega_1)$ are Weinstein trisections and that $\omega_0$ and $\omega_1$ are isotopic via a family $\omega_t$ such that $[\omega_t]\in H^2(X;\R)$ is independent of $t$. Does there exist a common stabilization $\TT'$ such that $(X,\TT',\omega_0)$ and $(X,\TT',\omega_1)$ are isotopic through Weinstein trisections?
\end{question}

More generally, we can ask whether Gay-Kirby's common stabilization theorem holds in the Weinstein setting.
	
\begin{question}\label{ques:stab}
	Suppose $(X,\TT,\omega)$ and $(X,\TT',\omega')$ are Weinstein trisections such that $\omega=\omega'$ (or, alternatively, such that $\omega$ and $\omega'$ are isotopic).  Is there a Weinstein trisection $(X,\TT'',\omega)$ such that $\TT''$ is a stabilization of both $\TT$ and $\TT'$?
\end{question}

A long-standing point of interest in symplectic geometry is the question of how to compare different symplectic structures on a fixed 4--manifold. The ability to stabilize Weinstein trisection provides a possible avenue to address this question in the following way.  Given two symplectic structures $\omega_1$ and $\omega_2$ on $X$, we can find Weinstein trisections $\TT_1$ and $\TT_2$ for $X$ that are compatible with $\omega_1$ and $\omega_2$, respectively.  By stabilizing, we obtain Weinstein trisections $\TT_1'$ and $\TT_2'$ that are isotopic as trisections, but may or may not be isotopic via a family of Weinstein trisections.  Thus, this allows us to compare the symplectic structures $\omega_1$ and $\omega_2$ on $X$ relative to a fixed trisection $\TT'$ of $X$.  
To this end, Questions~\ref{ques:isot},~\ref{ques:stabisot}, and~\ref{ques:stab} are relevant.

	\section{Singular bridge trisections and branched coverings}
	\label{sec:sing}
	
	In this section, we discuss singular surfaces in $4$--manifolds and their branched coverings, along with their compatibility with symplectic structures and trisections.

	\subsection{Singular bridge trisections}
	\label{subsec:sing-bt}
	\ 
	
	Let $X$ be a closed, orientable, smooth 4--manifold.  
	\begin{definition} We say $\cS\subset X$ is a \emph{singular surface} provided that
	\begin{enumerate}
		\item $\cS$ is the image of a smooth immersion away from finitely many critical points;
		\item all multiple points of the immersion are transverse double points; and
		\item in a small neighborhood of each critical value, $\cS$ is diffeomorphic to a cone on a knot in $S^3$.
	\end{enumerate}
	\end{definition}

\begin{remark}
\label{rmk:cone}
	We clarify what we mean by a \emph{cone}. Under an identification of the neighborhood of the singular point with $B^4$ that sends the singular point to the origin, the intersection of the singular surface with concentric spheres $S^3_r$ is a smooth knot. These knots will all be isotopic if we identify the concentric $3$--spheres by rescaling.
	Note that cones are not unique, depending on the way in which concentric cross-sectional knots are related by isotopy; there is generally \emph{not} a diffeomorphism of $B^4$ relating these different ways of coning along the same knot isotopy class.
	
	Results in this paper that depend only on smooth-topological considerations, such as the results of Section~\ref{sec:pull-back} are valid regardless of how the cone is formed; the main relevant fact will be that the branched covers of the concentric 3--spheres along their cross-sectional slices are 3--spheres.  In particular, the linear cone will always suffice in this setting.
	
	For symplectic surfaces, we will constrain the cones much more precisely. We will only look at symplectic surfaces where the singular points have complex models (except potentially negative double points coming from the immersion). In particular, each branch of the singularity will have a well-defined limiting tangent space. (See Remark~\ref{rmk:tangent}.)

	Note that although the linear cone on a knot may serve perfectly well as the branch locus for a branched covering in the smooth trisection setting, it does not suffice for the symplectic setting. In fact, for branched coverings, when working in the symplectic setting, we will only need to consider critical points that are the cone on a right-handed trefoil, and we will fix an explicit complex model for this cone. We call the singularity with this model a (simple) \emph{cusp}.  See Example~\ref{ex:cone-br}.
	

\end{remark}

	For a singular surface to be \emph{symplectic}, we want its tangent spaces to be symplectic subspaces where they are defined, and at singular points, we want specific local models.  Note, that throughout this article, a symplectic \emph{surface} will have real dimension two; we will refer to the ambient manifold as a symplectic 4--manifold.
	
	\begin{definition}
		A \emph{symplectic singular surface} in a symplectic $4$--manifold $(X,\omega)$ is the image $Q=g(S)$ of a map $g\colon S \to X$, with a finite set of points $P=\{p_1,\cdots, p_n\}\subset S$, where $S$ is an abstract surface, satisfying the following properties. 
			\begin{enumerate}
				\item $g|_{S\setminus P}$ is a smooth immersion.
				\item $\omega$ is positive on $dg_p(T_pS)$ for $p\in S\setminus P$.
				\item For each $i\in \{1,\cdots, n\}$, there exists a neighborhood $B_i$ of $g(p_i)$ in $X$, and a diffeomorphism $\Psi\colon B_i \to U_i\subset (\C^2,\omega_{std})$ such that if $N_i$ is the connected component of $g^{-1}(B_i)$ in $S$ containing~$p_i$, then $\Psi\circ g(N_i)\subset U_i\subset \C^2$ is cut out by a complex polynomial. (Each branch of the singular surface is smoothly modeled on a complex algebraic curve.)
				\item For each $p\in P$, let $T_p$ be the limiting tangent space $\displaystyle T_p=\lim_{x\to p} dg_{x}(T_{x}S)$; then, $\omega$ is positive on~$T_p$.
			\end{enumerate}
	\end{definition}
	
	\begin{remark}\label{rmk:tangent}
		Condition (4) makes sense as a consequence of condition (3), because each branch of a complex plane curve singularity has a well-defined limiting tangent space. This follows from the existence of the Puiseux parametrization of the branch: namely there exists a smooth change of coordinates (where $x$ and $y$ will be viewed as complex coordinates each containing two real parameters), such that the curve is locally parameterized as $x(t)=t^n$, $y(t) = \sum_{r=n}^\infty a_r t^r$. Since 
		$$\frac{dy}{dx}(t) = a_n+\sum_{r=n+1}^\infty \frac{ra_r}{n}t^{r-n}, $$
		we have that the tangent spaces near the critical point limit to the tangent line $y=a_n x$. See~\cite[Chapter 2]{Wall} for more details.
	\end{remark}

	A 1--parameter family $Q_t=g_t(S)$ of symplectic singular surfaces in $(X,\omega)$ such that the links of the singular points remain smoothly isotopic knots/links throughout the family, will be called a \emph{topologically equisingular symplectic isotopy}. 
	
	Note that the analytic type of the singularities may change. In particular, if $q_0\in Q_0$ and $q_1\in Q_1$ are corresponding singularities in such a family, there need \emph{not} exist neighborhoods $q_i\in U_i\subset X$ and a symplectomorphism $\Phi\colon U_0\to U_1$ such that $\Phi(Q_0\cap U_0) = Q_1\cap U_1$. Therefore, there need not be an ambient symplectic isotopy inducing this isotopy of surfaces. In other words, there may not be a family of symplectomorphisms $\Psi_t\colon(X,\omega)\to (X,\omega)$ such that $\Psi_t(Q_0)=Q_t$. However, the next lemma shows that the local singularities are the only obstruction to finding such an ambient symplectic isotopy.
	
	\begin{lemma} \label{l:sympliso}
		Let $Q_t=g_t(S)$ be a topologically equisingular symplectic isotopy, for $t\in [0,1]$.
		Then there exist small neighborhoods $\{U_t^i\}_{i=1}^N$ of the singular points of $Q_t$ and a family of symplectomorphisms $\Psi_t\colon(X,\omega)\to (X,\omega)$ such that
		$$\Psi_t(Q_0)\setminus \left( \bigcup_{i=1}^N U_t^i \right) = Q_t\setminus \left( \bigcup_{i=1}^N U_t^i \right).$$
	\end{lemma}
	
	\begin{proof}
		The case when $Q_t$ is a smooth symplectic submanifold is proven in \cite[Proposition 4]{AurouxIsotopy} using a Moser argument (in this case there are no neighborhoods $U_t^i$ of singular points). We review this argument inserting the necessary modifications for the singular case.
		
		The first step is to find neighborhoods $N_t$ of $Q_t$ and symplectomorphisms $\psi_t\colon N_0 \to N_t$ such that $\psi_t(Q_0)$ agrees with $Q_t$ outside of the neighborhoods $U_t^i$. Outside of any small neighborhood of the singular points, where $Q_t$ is smooth and symplectic, it has a standard symplectic neighborhood (because the symplectic normal bundle is uniquely determined by the smooth normal bundle). This defines $N_t$ and $\psi_t$ away from the singular points. Each sufficiently small neighborhood of a singular point has a Darboux chart (though the precise model for the intersection of this Darboux chart with $Q_t$ may vary). These standard pieces can be glued together compatibly along a contactomorphism of $(S^3,\xi_{std})$ along the boundary of the Darboux chart. If the topological singularity type is fixed for all $t$, the transverse link of the singularity is the same up to transverse isotopy for all $t$. Since any transverse isotopy can be realized by a contact isotopy, we can assume that the gluing aligns the smooth piece of $Q_t$ with the singular part of $Q_t$. Note that every contactomorphism of $(S^3,\xi_{std})$ extends to a symplectomorphism of $(B^4,\omega_{std})$ so while we may not have control over the precise image of $Q_t$ near its singular points, we do know that we can extend the standard neighborhood on the smooth part to a symplectomorphism $\psi_t\colon N_0\to N_t$ such that $\psi_t(Q_0)\setminus \left(\cup_i U_t^i\right) = Q_t\setminus \left(\cup_i U_t^i\right)$.
		
		Next, we extend $\psi_t$ to a diffeomorphism $F_t\colon X\to X$ such that $F_0$ is the identity. Because this extension was arbitrary, we do not expect $\omega_t:=F_t^*(\omega)$ to equal $\omega$ outside of the neighborhood $N_0$ for $t\neq 0$, but we will use Moser's trick to construct another family of diffeomorphisms $\tau_t\colon X\to X$ such that $\tau_t^*(F_t^*\omega) = \omega$, such that $\tau_t(F_t(Q_0))=F_t(Q_0)$. Note that $F_t(Q_0)$ agrees by construction with $Q_t$ outside of the neighborhoods $U_t^i$.
		
		In Moser's trick, $\tau_t$ is defined as the integral flow of a vector field $V_t$ such that $d(\iota_{V_t}\omega_t)=-\frac{d\omega_t}{dt}$. Moser's method verifies that this condition on $V_t$ is sufficient to ensure that $\tau_t^*\omega_t=\omega$, but we need to know that there exists such a $V_t$ whose flow preserves $Q_0$. By non-degeneracy, finding $V_t$ is equivalent to finding $\alpha_t=\iota_{V_t}\omega_t$ such that $d\alpha_t=-\frac{d\omega_t}{dt}$. We want $V_t$ to vanish at the singularities of $Q_0$ and to be tangent to $Q_0$ at smooth points. Since $\omega_t=\omega$ in a neighborhood of $Q_0$ (because $F_t$ agrees with the symplectomorphism $\psi_t$ here), these conditions are equivalent to asking that $\alpha_t$ vanishes at the singularities of $Q_0$ and $T_pQ_0^{\perp_\omega}\subset \ker(\alpha_t)_p$ at each smooth point $p\in Q_0$. We know that $-\frac{d\omega_t}{dt}$ is zero in cohomology so it has some family of primitives $\beta_t$. We will modify $\beta_t$ to a $1$--form $\alpha_t$, differing from $\beta_t$ by an exact form and satisfying the desired tangency and vanishing criteria along $Q_0$.
		
		First, choose a $1$--form $\delta_t$ on $Q_0$ such that $\delta_t=0$ in an $2\varepsilon$--neighborhood of the singularities such that $i^*[\beta_t]=[\delta_t]\in H^1(D_0;\R)$ where $i\colon Q_0\to N_0$ is the inclusion. 
		Let $\pi\colon N_0\setminus \left(\cup_i B^i \right) \to Q_0\setminus \left(\cup_iU^i_0 \right)$ be the symplectic orthogonal projection of the neighborhood of $Q_0$ to $Q_0$ outside a $\varepsilon$--neighborhood of the singularities (we can extend $\pi$ arbitrarily near the singularities). 
		In the $\varepsilon$--neighborhoods of the singularities, let $\gamma_t=0$ and elsewhere let $\gamma_t=\pi^*\delta_t$. 
		Then $\gamma_t$ are smooth $1$--forms, $i^*[\gamma_t]=[\delta_t]=i^*[\beta_t]$ so $[\gamma_t]=[\beta_t]$ since $i$ is a homotopy equivalence. Additionally, $\gamma_t$ satisfies $(T_xC_0)^{\perp_\omega}\subseteq \ker\gamma_t$ at all points and $\gamma_t=0$ at the singularities of $Q_0$. 
		Now write $\gamma_t-\beta_t=df_t$ for functions $f_t:U_0\to \R$, and choose arbitrary extensions $\bar{f}_t\colon M\to \R$. 
		Then $\alpha_t=\beta_t+d\bar{f}_t$ agrees with $\gamma_t$ in $U_0$ (so $\alpha_t$ has the required vanishing and kernel conditions along $C_0$) and it satisfies $d\alpha_t = d\beta_t = -\frac{d\omega_t}{dt}$ as required.
	\end{proof}
	
	Next we explain how singular surfaces can interact well with a trisection.  For this, we require a 4--dimensional analog of a trivial tangle.

Let $Y=\#^k(S^1\times S^2)$ and let $Z = \natural^k(S^1\times B^3)$, so $Y=\partial Z$.
Let $L = L_1\sqcup\cdots\sqcup L_n$ be a link in $Y$ given as the split union of non-split links $L_i$ such that each $L_i$ is contained in a 3--ball in $Y$.
Let $\cD\subset Z$ be a collection of disks with $\partial\cD = L$.

We call $\cD$ a \emph{singular disk-tangle} for $L$ if
\begin{enumerate}
	\item Each component of $\cD$ with unknotted boundary is a smooth, properly embedded, boundary parallel disk;
	\item For each component $D$ of $\cD$ with knotted boundary $K$, there is a 4--ball $B\subset Z$ such that $B\cap S^3$ is a 3--ball containing $K$ and $D$ is a cone $K$ in $B$; and
	\item the components of $\cD$ are disjoint when their boundaries are split and intersect transversely -- away from all cone points -- otherwise.
\end{enumerate}
For example, if $L$ were the split union of a the torus link $T(9,6)$, whose components are three trefoils, and the Hopf link $T(2,1)$, then $\cD$ would consist of three cones on trefoils, which intersect pairwise transversely in six points, split union two smoothly embedded disks, which intersect transversely in a single point.
Note that this is a generalization of the notion of a trivial disk-tangle appearing elsewhere in the literature~\cite{CK,LM-Complex,MZ-Bridge1,MZ-Bridge2}.
		
\begin{definition}
	Let $X$ be a 4--manifold equipped with a trisection $\TT$, and let $\cS\subset X$ be a singular surface.  We say that $\cS$ is in \emph{bridge trisected position} with respect to $\TT$ if, for each $\lambda\in\Z_3$,
	\begin{enumerate}
		\item $\cT_\lambda = H_\lambda\cap\cS$ is a trivial tangle, and
		\item $\cD_\lambda = Z_\lambda\cap \cS$ is a singular disk-tangle for the link $K_\lambda = \cT_\lambda\cup\cT_{\lambda+1}$.
	\end{enumerate}
	We refer to the $\cD_\lambda$ as \emph{patches} and to the $\cT_\lambda$ as \emph{seams}.
	The decomposition
		$$(X,\cS) = (Z_1,\cD_1)\cup(Z_2,\cD_2)\cup(Z_3,\cD_3)$$
		is called a \emph{singular bridge trisection}.
\end{definition}
	
	A straight-forward adaptation of the techniques of~\cite{MZ-Bridge2} to the setting of singular surfaces reveals that any singular surface can be put in bridge trisected position with respect to any trisection. (The main idea: Treat the singular points as minima and follow the Morse-theoretic construction of the bridge trisected position.)
	In Section~\ref{sec:bridge}, we prove a stronger result for symplectic braided surfaces, showing that they can be \emph{symplectically} isotoped to lie in bridge trisected position in $\CP^2$.

	\subsection{Singular branched coverings}
	\label{subsec:sing-br}
	\ 
	
	We refer the reader to~\cite{Zuddas} for a nice exposition of branched coverings. Most branched coverings considered here will be \emph{simple} (meaning they are modeled by an involution near non-singular branch points), since these are the coverings produced by Auroux and Katzarkov~\cite{Auroux,Auroux-Katzarkov}. However, we will also consider some cyclic branched coverings of higher order when studying particularly nice examples (cf. Definition~\ref{def:sing-br} and Section~\ref{sec:Examples}).  Moreover, some results, such as Theorem~\ref{prop:branch-tri}, hold for more general branched coverings.  We begin with the following general definition.
	
	\begin{definition}
		\label{def:sing-br}
		Let $X$ be a 4--manifold and let $\cS$ in $X$ be a singular surface.  A proper map $f\colon \widetilde X\to X$ is called a \emph{singular branched covering} of $X$ along $\cS$ if the following two conditions hold:
		\begin{enumerate}
			\item Away from the singular points in $X$, $f$ is a branched covering map.
			\item In a 4--ball neighborhood $B\subset X$ of a singular point, $f^{-1}(B) = \widetilde B_1\sqcup\cdots\sqcup\widetilde B_n$ is the disjoint union of 4--balls such that $f\vert_{\widetilde B_i}\colon \widetilde B_i\to B$ is either a diffeomorphism or the cone on a branched covering map from $S^3$ to itself.  
		\end{enumerate}
	\end{definition}
	
	Note that the existence of a singular branched covering of $X$ branched along a singular surface $\cS$ restricts the singularities of $\cS$ significantly.  In particular, the link of any singularity must be a knot or link that admits a branched covering to the disjoint union of copies of $S^3$.  On the other hand, there are many examples of such links, including 2--bridge links.  We will now discuss explicitly the most important cases for the present work.
	
	\begin{examples}
		\label{ex:cone-br}
		We describe some important examples of singular branched coverings in coordinates.  Since we will mainly consider \emph{simple} singular branched coverings in this paper, the following three examples exhaust the relevant local models describing our singular branched covering maps.
	\begin{enumerate}
		\item \textbf{(simple branch point)} First, consider the standard branching model, described by the map $f\colon \C^2\to\C^2$ given in complex coordinates by $f(x,y)=(x^2,y)$.
		The 3--sphere $S_\varepsilon = \{|z|^2+|w|^2=\varepsilon\}$
		in the co-domain has preimage $f^{-1}(S_\varepsilon)=\{|x|^4+|y|^2=\varepsilon\}$ diffeomorphic to a 3--sphere, which can be seen by changing coordinates by the scaling map
		$$x \mapsto \frac{x}{\sqrt{|x|}}.$$
		The singular points of $f$ occur when $x=0$. The intersection of $\{x=0\}$ with $f^{-1}(S_\varepsilon)$ is an unknot.  Thus, $f\vert_{S_\varepsilon}$ represents the 2--fold cyclic covering of $S^3$ over itself, branched along the unknot, and $f$ can be thought of as the cone of the map $f_{S_\varepsilon}$, though in this case, the cone can be treated as a smooth, trivial disk bounded by the unknot.
		
		\item \textbf{(cusp)} Next, we describe the model for the cusp, which is given in complex coordinates by the function
		$$f(x,y) = (x^3-xy,y).$$
		The critical points of $f$ occur where the derivative
		$$df_{(x,y)} = \left[ \begin{array}{cc} 3x^2-y & -x\\ 0 & 1 \end{array} \right]$$
		fails to have full rank.

		This occurs when $3x^2-y=0$. Therefore, the critical set of this map is $\{y=3x^2\}$. The critical values of $f$ are thus the image of this set, namely the points $(-2x^3, 3x^2)$ for $x\in\C$. With coordinates $(z,w)$ on the codomain, the critical values are the zero locus of the polynomial $27z^2-4w^3$. Thus the branch locus is parameterized as $\{(-2c^3,3c^2) \mid c\in \C \}$ and has a singular point at $(0,0)$. Observe that for each value $c$, $f^{-1}(-2c^3, 3c^2) = \{(c,3c^2),(-2c,3c^2)\}$.  The points $(c,3c^2)$ are critical points and have multiplicity two. On the other hand, the points $(-2c,3c^2)$ (the locus $y=\frac{3}{4}x^2$) are not critical points when $c\neq 0$, but they do project to critical values in the branch locus. This occurs because the branched covering is irregular.
		
		To view this as a cone over a branched covering of $S^3$, consider again the 3--sphere $S_{\varepsilon}$, and let $\widetilde S_\varepsilon = f^{-1}(S_\varepsilon) =\{|x^3-xy|^2 + |y|^2 = \varepsilon\}$. The singular curve $\{27z^2-4w^3=0\}$ intersects $S_{\varepsilon}$ along the right-handed trefoil $K$, and $f$ is the cone of the restriction $f\vert_{\widetilde S_\varepsilon}\colon \widetilde S_\varepsilon \to S_\varepsilon$.
		
		In a neighborhood of any point on $K\subset S_\varepsilon$, $f\vert_{\widetilde S_\varepsilon}$ is given by the standard branching model described above.  It follows that $f\vert_{\widetilde S_\varepsilon}$ is a 3--fold, simple covering of $S_\varepsilon$, branched along $K$.  Such a covering is determined by a map $\rho\colon\pi_1(S_\varepsilon\setminus\nu(K))\twoheadrightarrow \mathscr S_3$ sending meridians of $K$ to transpositions.  (See~\cite{Zuddas} for definitions and details.)  Since such surjections are unique up to conjugation for the trefoil, it follows that $f\vert_{\widetilde S_\varepsilon}$ is the (only) irregular 3--fold branched cover over the trefoil.  Thus $f^{-1}(S_\varepsilon)\cong S^3$, as desired; the cusp map is the cone on this branched covering.
		
		\item \textbf{(node)} Finally, consider the quotient map $f\colon \C^2_1\sqcup\C^2_2 \to \C^2$ defined by $f(x_1,y_1) = (x_1^2,y_1)$ for $(x_1,y_1)\in \C^2_1$ and $f(x_2,y_2)=(x_2,y_2^2)$ for $(x_2,y_2)\in \C^2_2$. This can also be thought of as the quotient map by the equivalence relation where $(x_1,y_1)\sim (-x_1,y_1)$, $(x_2,y_2)\sim(x_2,-y_2)$, and $(x_1,y_1)\sim (x_2,y_2)$ if $x_1^2 = x_2$ and $y_1=y_2^2$ for any $(x_1,y_1)\in \C^2_1$ and $(x_2,y_2)\in \C^2_2$.  Over points $(z,w)\in\C^2$, this map is four-to-one when $z\not=0$ and $w\not=0$, is three-to-one when either $z=0$ or $w=0$ (but not both), and is two-to-one when $(z,w) = (0,0)$.
		
		We can view this map as the disjoint union of cones of branched coverings from $S^3$ to $S^3$ as follows.  For $i=1,2$, the restriction $f_i$ of $f$ to $\C_i$ is simply the map from Example~(1) above (up to swapping the coordinates in $\C^2_2$).  Each of these is a cone, so $f$ is the disjoint union of cones on these maps 
		
		Finally, note that in this case the link $L_\varepsilon = S_\varepsilon\cap\{zw=0\}$ is simply the Hopf link, the lift $f^{-1}(S_\varepsilon)$ is the disjoint union of two 3--spheres, and the restriction $f\vert_{f^{-1}(S_\varepsilon)}$ is a 4--fold covering of $S^3$, branched along the Hopf link. The preimage of the Hopf link is a pair of Hopf links in the two covering 3--spheres, and each component of the link downstairs is double covered by one component upstairs and single covered by one component upstairs. This covering is determined by the map $\rho\colon \pi_1(S_\varepsilon\setminus\nu(L_\varepsilon))\to \mathscr S_4$ that sends the meridians of the Hopf link to commuting transpositions, say $(1\,2)$ and $(3\,4)$.  Again, this map is unique up to conjugation, but this time it is not surjective.  The non-surjectivity of $\rho$ corresponds to the fact that the branched covering is disconnected.
		
		Note: Reversing the orientation on $\C^2$ to disagree with the complex orientation changes this to the model for a negative node.
	\end{enumerate}	
	\end{examples}

	Our use of singular branched coverings here comes from their connection with symplectic manifolds. The following definition captures the key result that we need from Auroux's work.
	
	\begin{definition}\label{def:ssbc}
		A map $f\colon (X,\omega)\to (\cptwo,\omega_{FS})$ of degree $k$ is a \emph{symplectic singular branched covering} if near every point it is modeled smoothly on one of the three following maps.
		\begin{enumerate}
			\item local diffeomorphism: $(z_1,z_2)\mapsto (z_1,z_2)$
			\item cyclic branched covering: $(z_1,z_2)\mapsto (z_1^d,z_2)$
			\item cusp covering: $(z_1,z_2)\mapsto (z_1^3-z_1z_2,z_2)$
		\end{enumerate}
		Moreover, we require that the branch locus $R\subset X$ (where $f$ cannot be modeled by the local diffeomorphism) and its image $f(R)$ be (singular) symplectic surfaces, the co-homology class of the pull-back of the Fubini-Study form satisfies $[f^*\omega_{FS}]=k[\omega]\in H^2(X)$, and the $2$--form $\widetilde{\omega}_t:= tf^*\omega_{FS}+(1-t)k\omega$ is a symplectic form for all $t\in[0,1)$.
	\end{definition}
	
	The branched covering and cusp covering models were examined in detail in Examples~\ref{ex:cone-br} above.
	Note that these models are local in the domain of the map, and the images of these local charts in the co-domain $\cptwo$ may overlap. In particular, the image of the branching locus may self-intersect (either positively or negatively), as in the node in Example~\ref{ex:cone-br}(3).

	\begin{theorem}\cite[Theorem~1 and Proposition~11]{Auroux} \label{thm:auroux}
		Every closed symplectic $4$--manifold admits a symplectic singular branched covering to $\cptwo$.
	\end{theorem}
	
	\begin{remark}
		Theorem~1 of~\cite{Auroux} states that every closed symplectic $4$--manifold admits an \emph{$\varepsilon$--holomorphic} singular covering. The $\varepsilon$--holomorphic condition is related to an almost complex structure $J$ which is compatible with the symplectic form. This condition means that the coordinate charts used to identify $f$ locally with the models take $J$ very close to the standard integrable complex structure on $\C^2$. For the purposes of this paper, we want to deal directly with the symplectic form instead of a compatible almost complex structure.
		If $\varepsilon<1$, the branch locus $R$ and its image $f(R)$ will be (singular) symplectic surfaces. In the paragraphs preceding \cite[Proposition 11]{Auroux}, it is noted that the cohomological condition $[f^*\omega_{FS}]=k[\omega]$ holds for Auroux's method of producing of branched covers. Auroux also proved that an $\varepsilon$--holomorphic singular covering also satisfies the symplectic condition that $\widetilde{\omega}_t:=tf^*\omega_{FS}+(1-t)k\omega$ is symplectic for $t\in[0,1)$  as \cite[Proposition 11]{Auroux}. Therefore Theorem \ref{thm:auroux} follows from combining these results from ~\cite{Auroux}. By Moser's theorem $(X,\widetilde{\omega}_t)$ is symplectomorphic to $(X,k\omega)$ for $t<1$. The pull-back form $f^*\omega_{FS}$ is not itself symplectic because it is degenerate along the branch locus. However, the condition that $\widetilde{\omega}_t$ is symplectic for $t<1$ shows that $f^*\omega_{FS}$ is arbitrarily close to a symplectic form that is symplectomorphic to $(X,k\omega)$. In fact, there is an explicit small perturbation of $f^*\omega_{FS}$ which is localized near the branch locus that yields a symplectic form directly from the smooth map $f$. We will discuss and use this perturbation in detail in Section~\ref{sec:lifting}.
	\end{remark}
	
	\begin{remark}
		Auroux actually shows that such symplectic singular branched coverings exist without needing the branched covering model for $d>2$. The $d=2$ (simple) branched covering model and the cusp covering model are the generic singularities that occur in holomorphic maps from one complex $2$--dimensional manifold to another. We allow the non-generic $d>2$ branched covering model because we can sometimes find a map $f$ with significantly simpler branch locus if we allow the branched covering model with higher values of $d$. Examples of this are included in Section~\ref{sec:Examples}.
	\end{remark}

\section{Singular branched covers pull back trisections}
\label{sec:pull-back}

In this section we prove that the pull back of a trisection under a singular branched cover is a trisection. This part is purely about the smooth topology and holds for general singular branched coverings. We will address the symplectic aspects in Section~\ref{sec:lifting}.
	
	\begin{theorem}
		\label{prop:branch-tri}
		Let $X$ be a 4--manifold, and let $\TT$ be a trisection on $X$ given by the decomposition
		$$X = Z_1\cup Z_2\cup Z_3.$$
		Let $f\colon \widetilde X\to X$ be a singular branched covering with branch locus the singular surface $\cS$ in $X$.
		Assume that $\cS$ is in bridge trisected position with respect to $\TT$.
		Then, letting $\widetilde Z_\lambda = f^{-1}(Z_\lambda)$, the decomposition
		$$\widetilde X = \widetilde Z_1\cup \widetilde Z_2\cup \widetilde Z_3$$
		is a trisection of $\widetilde X$, provided $\widetilde X$ is connected.
	\end{theorem}
	
	\begin{proof}
		Let $H_\lambda = Z_\lambda\cap Z_{\lambda-1}$, and let $\Sigma = \partial H_\lambda$.  Let $\widetilde H_\lambda = f^{-1}(H_\lambda)$, and let $\widetilde\Sigma = f^{-1}(\Sigma)$. Note that $\widetilde H_\lambda = \widetilde Z_\lambda\cap \widetilde Z_{\lambda-1}$ and $\widetilde \Sigma = \widetilde Z_1\cap \widetilde Z_2 \cap \widetilde Z_3$.
		
		Consider the restriction $f\vert_{\widetilde\Sigma}\colon\widetilde\Sigma\to\Sigma$.  This map is a covering of the closed, orientable surface $\Sigma$, branched along the bridge points $\bold x = \Sigma\cap\cS$.  It follows that $\widetilde\Sigma$ is a closed, orientable surface.  A priori, $\widetilde\Sigma$ may be disconnected, however.
		
		Next, consider the restriction $f\vert_{\widetilde H_\lambda}$.  This map is a covering of the handlebody $H_\lambda$, branched along the seams $\cT_\lambda = H_\lambda\cap\cS$, which form a trivial tangle $(H_\lambda,\cT_\lambda)$.  Such a trivial tangle is diffeomorphic to $(S,\bold y)\times I$, where $S$ is a compact, orientable surface with boundary and $\bold y\in\text{Int}(S)$ is a collection of points.  (This is the definition of a trivial tangle.)  It follows that $f\vert_{\widetilde H_\lambda}$ has the structure of a covering of $S$, branched along the points $\bold y$, crossed with $I$.  Since such a cover of $S$ is, again, a compact, orientable surface with boundary, we see that $\widetilde H_\lambda$ is a handlebody (or at least a disjoint union of handlebodies).  Note that $\partial \widetilde H_\lambda = \widetilde\Sigma$, so $\widetilde \Sigma$ is connected if and only if the $\widetilde H_\lambda$ are.
		
		Finally, consider the restriction $f\vert_{\widetilde Z_\lambda}$. It remains to see that $\widetilde Z_\lambda$ is a 4--dimensional 1--handlebody.  Break $Z_\lambda$ up into two pieces based on the components of the disk-tangle $\cD_\lambda$ as follows.  Recall that each connected component of $\cD_\lambda$ (as a subset of $Z_\lambda$) has boundary a knot or link. For each connected component $D$, consider a properly embedded 4--ball $B_D$ that splits that component from the rest of $\cD_\lambda$.  Let $Z^\text{branch}$ denote the disjoint union of these  4--balls.  Let $Z^\text{rest}$ denote $Z_\lambda\setminus\text{Int}(Z^\text{branch})$.  Note that $Z^\text{rest}$ is diffeomorphic to $Z_\lambda$; since each component of $Z^\text{branch}$ is a 4--ball with half of its boundary on $\partial Z_\lambda$, the removal of $Z^\text{branch}$ from $Z_\lambda$ doesn't change its diffeomorphism type.  Let $\widetilde Z^\text{branch} = f^{-1}(Z^\text{branch})$, and let $\widetilde Z^\text{rest} = f^{-1}(Z^\text{rest})$.
	
		By definition, the restriction $f\vert_{f^{-1}(B_D)}\colon f^{-1}(B_D)\to B_D$ is a map of a disjoint union of 4--balls $\widetilde B_1,\ldots, \widetilde B_l$ to the 4--ball $B_D$, such that each map $f_j\colon\widetilde B_j\to B_D$ is given as the cone on a (possibly trivial) covering of $S^3$ by itself, branched along a (possibly empty) subset of the components of the link $\partial D$. Note that the 3--sphere boundary of $B_D$ is divided into two 3--balls, $\partial B_D = K\cup L$ where $K=B_D \cap Z^\text{rest}$ and $L=B_D\cap \partial Z_\lambda$. Note that $K$ is disjoint from the branch locus and $L$ contains the link $\partial D$. Then for each connected component $\widetilde B_j$ of $f^{-1}(B_D)$, the boundary 3--sphere of $\widetilde B_j$ is decomposed into $\widetilde K_j=\widetilde B_j\cap f^{-1}(K)$ and $\widetilde L_j=\widetilde B_j\cap f^{-1}(L)$. Because $K$ is a 3--ball disjoint from the branch locus, $\widetilde K_j$ is a disjoint union of 3--balls. $\widetilde L_j$ is a 3--sphere with a collection of 3--balls removed.
		Now, each 4--ball $\widetilde B_j$ is attached to $\widetilde Z^\text{rest}$ along $\widetilde K_j$, a disjoint union of 3--balls. If $\widetilde K_j$ consists of a single 3--ball, then the attachment of $\widetilde B_j$ adds only a trivial collar to $\widetilde Z^\text{rest}$. If $\widetilde K_j$ consists of exactly two 3--balls, attaching $\widetilde B_j$ to $\widetilde Z^\text{rest}$ amounts to attaching a $4$--dimensional $1$--handle. In general, when $\widetilde K_j$ consists of $m$ 3--balls, the attachment of $\widetilde B_j$ to $\widetilde Z^\text{rest}$ is equivalent to attaching $m-1$ $4$--dimensional $1$--handles.

		It follows that $\widetilde Z_\lambda = \widetilde Z^\text{rest}\cup\widetilde Z^\text{branch}$ is obtained by attaching 4--dimensional 1--handles to $\widetilde Z^\text{rest}$.  We next show that $\widetilde Z^\text{rest}$ is a 4--dimensional 1--handlebody (or a disjoint union thereof), which will imply that $\widetilde Z_\lambda$ is a 4--dimensional 1--handlebody (or disjoint union thereof), as desired.
		
		The restriction $f\vert_{\widetilde Z^\text{rest}}\colon \widetilde Z^\text{rest}\to Z^\text{rest}$ is an unbranched covering of $Z^\text{rest}$.
		Since $Z^\text{rest}$ is a 4--dimensional neighborhood of a graph, $\widetilde Z^\text{rest}$ is also a 4--dimensional neighborhood of a graph.  It follows that $\widetilde Z^\text{rest}$, hence $\widetilde Z_\lambda$ is a 4--dimensional 1--handlebody, or a disjoint union thereof.
		
		To complete the proof, we must argue that $\widetilde\Sigma$, $\widetilde H_\lambda$, and $\widetilde Z_\lambda$ are each connected. However, we have shown that $\widetilde\Sigma$ is a disjoint union of closed surfaces; that $\widetilde H_\lambda$ is a 3--dimensional handlebody with $\partial\widetilde H_\lambda = \widetilde\Sigma$; and that $\widetilde Z_\lambda$ is a 4--dimensional 1--handlebody with $\partial\widetilde Z_\lambda = \widetilde H_\lambda\cup_{\widetilde\Sigma}\widetilde H_{\lambda+1}$.  (That the boundaries are as claimed follows from the fact that the restrictions $f\vert_{\widetilde\Sigma}$, $f\vert_{\widetilde H_\lambda}$, and $f\vert_{\widetilde Z_\lambda}$ are all proper maps.) Therefore, $\widetilde\Sigma$, $\widetilde H_\lambda$, $\widetilde Z_\lambda$, and $\widetilde X$ must have the same number of connected components, as desired.

	\end{proof}

	In light of this proposition, we call the trisection of $\widetilde X$ obtained by pulling back $\TT$ via $f$ the \emph{pullback trisection} and we denote it by $f^*(\TT)$.
	
	\begin{remark}
		We note that the trisection parameters $\widetilde g$ and $\bold{\widetilde k} = (\widetilde k_1,\widetilde k_2,\widetilde k_3)$ corresponding to $f^*(\TT)$ can be calculated in practice.  The core $\widetilde\Sigma$ is a branched covering of the core $\Sigma$, so $g(\widetilde\Sigma)$ can related to $g(\Sigma)$ via the Riemann-Hurwitz Formula by considering the local degrees of the covering $f$.  Since $\partial \widetilde Z_\lambda$ is a cover of $\partial Z_\lambda$, branched along the link $\partial \cD_\lambda$, the parameter $\widetilde k_\lambda$ is related to the parameter $k_\lambda$ by understanding each of the coverings of a disjoint union of copies of $S^3$ over $S^3$, branched along the split components of this link.  In the case that $f$ is a symplectic singular branched covering, these components are Hopf links and trefoils, so this computation can be done by knowing the number of each and the local degrees.
	\end{remark}

	\section{Putting symplectic braided surfaces in bridge position}
	\label{sec:bridge}
	
	Theorem~\ref{prop:branch-tri} proves that a branched cover of a singular surface in bridge trisected position pulls back the trisection downstairs to a trisection upstairs. We want to be able to combine this result with Auroux's theorem (Theorem~\ref{thm:auroux}) that every symplectic manifold is a symplectic singular branched cover over $\cptwo$. In order to combine these, we need to ensure that the symplectic branch locus of an Auroux branched covering is in bridge trisected position with respect to a Weinstein trisection on $\cptwo$.

	To do this, we start with a refinement of Theorem \ref{thm:auroux} that was proven by Auroux and Katzarkov \cite{Auroux-Katzarkov}, which ensures that we find a symplectic singular branched covering $f\colon X\to \cptwo$ with the additional property that the image of the branch locus sits in a nice position with respect to the standard linear pencil on $\cptwo$ (see the definition below for a braided surface). In this section we will show that we can symplectically isotope such a braided surface so that it is in bridge position with respect to a genus one Weinstein trisection on $\cptwo$. Therefore, up to a symplectomorphism of $\cptwo$ from Lemma~\ref{l:sympliso} realizing this isotopy, the Auroux-Katzarkov branched covering pulls back a Weinstein trisection of $\cptwo$ to a trisection of an arbitrary symplectic manifold. In section~\ref{sec:lifting}, we will show that in fact this lifted trisection has Weinstein sectors, and put everything together to prove our main theorem.

First we precisely define the set-up and definitions for the output of Auroux and Katzarkov's results.

We will work with the standard linear pencil on $\cptwo$ given by the map $\pi\colon \cptwo\setminus \{[0:0:1]\} \to \cpone$ defined by $\pi([z_0:z_1:z_2])=[z_1:z_2]$. We will often work in the affine chart where $z_0\neq 0$ with coordinates $(z_1,z_2)$. The restriction of $\pi$ to this chart is the map $\pi\colon \C^2\to \C$ where $\pi(z_1,z_2)=z_1$.

\setcounter{theorem}{-1}
\begin{remark}
\label{rmk:coordinates}

	In Example~\ref{ex:CP2-standard}, we introduced homogeneous coordinates $[z_1:z_2:z_3]$ on $\cptwo$ that were compatible with the convention of labeling the sectors $Z_\lambda$ of a trisection with $\lambda\in\{1,2,3\}$.  In what follows, we adopt the more traditional convention of homogeneous coordinates given by $[z_0:z_1:z_2]$.  The main motivation for this is to be consistent with the conventions of~\cite{Auroux} and~\cite{Auroux-Katzarkov}.  These two conventions are related by the cyclic relabeling $[z_1:z_2:z_3]\mapsto[z_3:z_1:z_2]$.  In the coordinates of Example~\ref{ex:CP2-standard}, the standard linear pencil described above would be the map $\pi\colon\cptwo\setminus\{[0:1:0]\}$ given by $\pi([z_1:z_2:z_3]) = [z_1:z_2]$, and one could adopt the same affine coordinates in either case.
	If one considers the labeling of the pieces of the trisection to be given by $\lambda\in\Z_3$, then we have the $0\equiv 3$, and this congruence respects all aspects of the development that follows.
\end{remark}
	
\begin{definition}[cf. Definition 1 of \cite{Auroux-Katzarkov}]
		A \emph{braided surface} in $\cptwo$ is a singular surface $Q=g(S)$ for $g\colon S\to \cptwo$ such that
		\begin{enumerate}
			\item $[0:0:1]\notin Q$;
			\item $Q$ is positively transverse to the fibers of the pencil $\pi$, except at finitely many points where it becomes non-degenerately positively tangent to the fibers (locally modeled on $w_2^2=w_1$ where $\pi$ is the projection to the $w_1$ coordinate);
			\item the only singularities of $Q$ are (simple) cusps diffeomorphically modeled by Example~\ref{ex:cone-br} and transverse double points (nodes) which may be either positive or negative self-intersections;
			\item the cusps are positively transverse to the fibers of $\pi$; and
			\item the ``special points'' (cusps, double points, and tangency points) all lie in distinct fibers of $\pi$.
		\end{enumerate}
		A \emph{symplectic braided surface} is a braided surface which is a symplectic singular surface.
	We will say that two (symplectic) braided surfaces are \emph{isotopic through (symplectic) braided surfaces} if there is a one parameter family of (symplectic) braided surfaces connecting them.

	\end{definition}

\begin{remark}
	Auroux and Katzarkov work with an approximately holomorphic version of braided surfaces, which they call ``quasiholomorphic curves".  For our purpose, the relevant fact is that  approximately holomorphic implies symplectic.
\end{remark}
	
	\begin{theorem}[\cite{Auroux-Katzarkov}] \label{thm:AK}
		Every closed symplectic $4$--manifold admits a symplectic singular branched covering, such that the image of the branch locus in $\cptwo$ is a symplectic braided surface.
	\end{theorem}
	
	This theorem goes beyond Auroux's original Theorem~\ref{thm:auroux} because when the branch locus in $\cptwo$ has negative double points (see Example~\ref{ex:cone-br}), it is not automatic that the surface can be assumed to be braided with respect to a linear pencil. In general it is not currently known whether there are symplectic $4$-manifolds such that all symplectic branched coverings to $\cptwo$ have branch locus with negative double points. However, this possibility cannot be ruled out by the techniques of~\cite{Auroux,Auroux-Katzarkov}.
	
	\begin{remark}  \label{smoothisbraided}
		 On the other hand, when there are no negative double points, any smooth symplectic surface, or symplectic surface whose singularities are sympletically modeled on complex curve singularities (which we will call an \emph{algebraically singular symplectic surface}), can be assumed to be braided with respect to a linear pencil on $\cptwo$ (so the additional work of~\cite{Auroux-Katzarkov} is not needed in this case). More specifically, any smooth (or algebraically singular) symplectic surface $Q$ can be realized as a $J$--holomorphic curve for some almost complex structure $J$ which is compatible with the standard symplectic form $\omega_{FS}$ on $\cptwo$. The space $\mathcal{J}^{\omega_{FS}}(Q)$ of such $J$ which make $Q$ $J$--holomorphic is contractible and in particular non-empty (the smooth case is standard, see \cite[Lemma 3.4]{GollaStarkston} for the generalization to the singular case). Moreover for any such $J\in \mathcal{J}^{\omega_{FS}}(Q)$, we obtain a pencil of $J$--holomorphic lines through the point $[0:0:1]$ by Gromov~\cite[Theorem 0.2.B]{Gromov}, and (after a generic perturbation of $Q$) $Q$ will be braided with respect to this pencil by positivity of intersections of $J$-holomorphic curves. Moreover, Gromov shows that there is a symplectomorphism taking this $J$--holomorphic pencil to the standard complex linear pencil. Therefore, up to a symplectomorphism of $\cptwo$ (which is always symplectically isotopic to the identity), every symplectic surface which is smooth or algebraically singular is braided. This observation combined with Theorem~\ref{thm:symplbridgepos}, which we will prove at the end of this section, will imply Theorem~\ref{thm:symp_isot}.
	\end{remark}

	Now we will show that we can symplectically isotope a symplectic braided surface into bridge position through symplectic braided surfaces. To keep the braided condition, we will need to keep track of the interaction of $Q$ with the fibration $\pi\colon \cptwo\setminus\{[0:0:1]\}\to \cpone$. To guide the reader, we start with an outline of the proof of our main result of this section schematically represented in Figure~\ref{fig:outline}.
	
	\begin{figure}
		\centering
		\includegraphics[scale=.5]{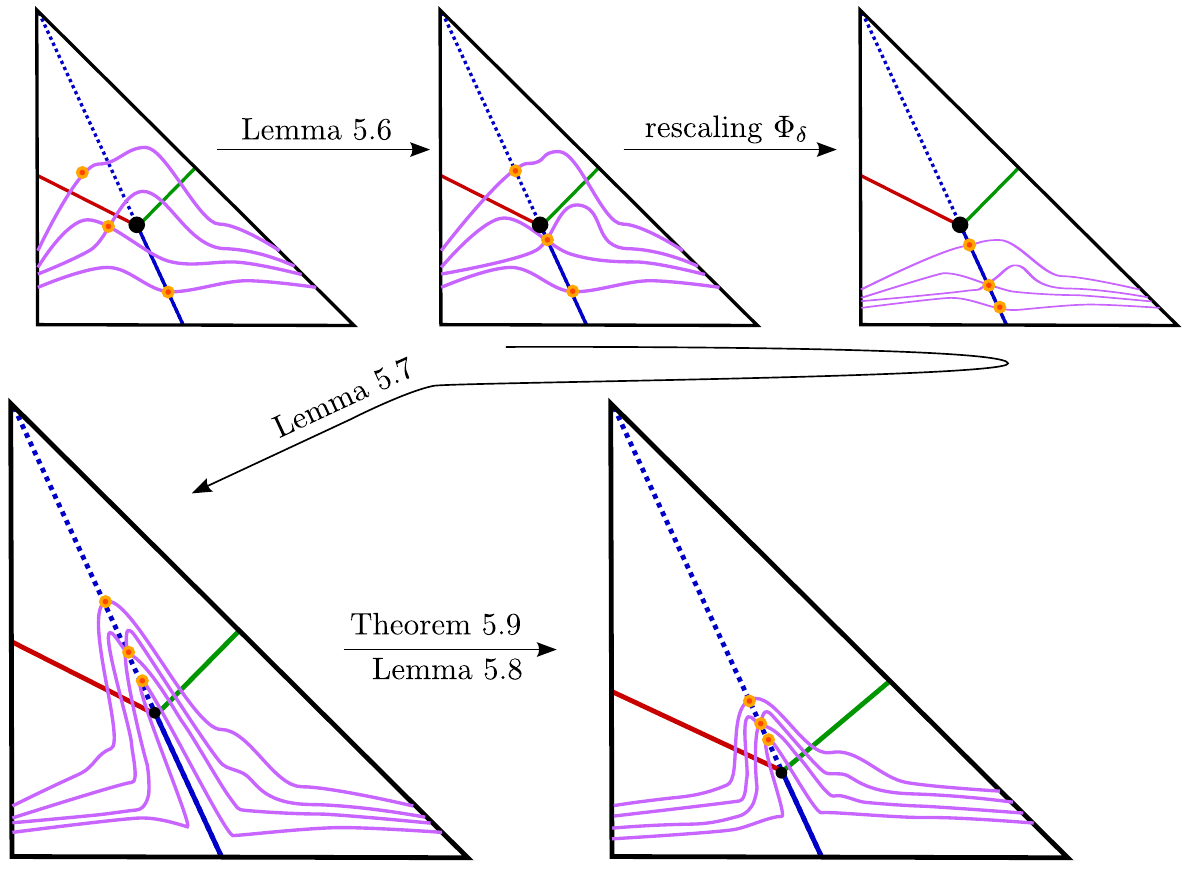}
		\caption{A rough schematic for the strategy to construct a symplectically isotopic braided in bridge trisected position.}
		\label{fig:outline}
	\end{figure}
	
	\begin{enumerate}
		\item In Lemma~\ref{l:equator}, we show how symplectic braided surfaces can be isotoped through symplectic braided surfaces so that critical values of $\pi\vert_Q$ lie on the equator of the copy of $\cpone$ given by $|z_1|=|z_0|$.  
		\item After a rescaling in the direction of the fibers of $\pi$, we can assume $Q$ becomes disjoint from $Z_3$ and the critical points of $\pi\vert_Q$ lie in $H_2$. We construct a class of fiber preserving rescaling maps below, which we will use throughout. These rescalings can be applied to the entire surface $Q$, or a subsurface. The key properties of these rescaling maps are studied in Lemma~\ref{l:rescale}.
		\item In Lemma~\ref{l:bridgepos}, we show that a symplectic braided surface whose critical values lie on the equator of $\cpone$ can be isotoped through braided surfaces in a controlled, fiber-preserving way to end in bridge trisected position with respect to a scaled version of the genus one Weinstein trisection $\TT_0$. This is where we make the key technical argument to get the surface into bridge position. That this stretching can be done through symplectic braided surfaces is not shown until later. 
		\item In Theorem~\ref{thm:symplbridgepos}, we take the isotopy of braided surfaces produced by Lemma~\ref{l:bridgepos}, and flatten it by a global rescaling to ensure the entire isotopy is symplectic. The end point of the isotopy will be in bridge trisected position with respect to a further rescaled version of the standard genus one Weinstein trisection $\TT_0$.
		\item In Lemma~\ref{l:scaletrisect}, we show that the globally rescaled version of $\TT_0$ is itself Weinstein, so our final symplectically isotoped braided surface is in bridge trisected position with respect to a genus one Weinstein trisection on $\cptwo$ as desired.
	\end{enumerate}

	Many of the ambient isotopies used in what follows involve stretching in a fiber preserving way.  We construct a useful class of such maps as follows.
	If $\phi\colon \cpone \to \R_{>0}$ is any smooth function, we obtain a fiber preserving diffeomorphism of $\cptwo$ by the map $\Phi\colon \cptwo\to \cptwo$, given by
	$$\Phi([z_0:z_1:z_2])=[z_0:z_1:\phi([z_0:z_1])z_2].$$
	If $Q$ is braided, then $\Phi(Q)$ is still braided. However, if $Q$ is symplectic, in general it is possible that $\Phi(Q)$ may no longer be symplectic. Note that when $\phi\equiv 1$, $\Phi$ is the identity. The case when $\phi$ is constant is particularly easy to understand.  In what follows, $S$ and $S'$ represent abstract surfaces and $g$ is an immersion except at isolated singular points, such that $g(S)$ and $g(S')$ are singular surfaces, possibly with boundary.

	\begin{lemma} \label{l:rescale}
		Let $\Phi_c([z_0:z_1:z_2])=[z_0:z_1:cz_2]$.
		\begin{enumerate}
			\item If $Q=g(S)$ where $g\colon S\to \cptwo\setminus\{[0:0:1]\}$ is a symplectic braided surface, and $c\leq 1$, then $\Phi_c(Q)$ is also a symplectic braided surface with respect to $\omega_{FS}$.
			\item If $g\colon S'\to \cptwo\setminus\{[0:0:1]\}$ is positively transverse to the fibers of $\pi$ and $S'$ is compact then there exists a $c$ sufficiently small such that $\Phi_c\circ g(S')$ is symplectic (with respect to $\omega_{FS}$). (Note that if $g$ is positively transverse to the fibers of $\pi$ its image is necessarily braided with no tangencies or cusps.)
		\end{enumerate}
	\end{lemma}	
	
	\begin{proof}
		In the chart where $z_0\neq 0$, the Fubini-Study symplectic form in affine coordinates $(z_1,z_2)$ is given by $\omega_{FS} = -\frac{1}{4}dd^{\C}(\log(1+|z_1|^2+|z_2|^2))$. Letting $z_j=x_j+iy_j$, and expanding out this $2$--form, we can calculate that		
		$$\Phi_c^*(\omega_{FS}) = (1-c^2)dx_1\wedge dy_1 +c^2 \omega_{FS}.$$
		Therefore 
		$$(\Phi_c\circ g)^*(\omega_{FS}) = g^*(\Phi_c^*\omega_{FS}) = g^*((1-c^2)dx_1\wedge dy_1+c^2\omega_{FS}) = (1-c^2)g^*(dx_1\wedge dy_1)+c^2g^*\omega_{FS}.$$
		
		Similarly in the chart where $z_1\neq 0$, in affine coordinates $(z_0,z_2)$, we find 
		$$(\Phi_c\circ g)^*(\omega_{FS}) = (1-c^2)g^*(dx_0\wedge dy_0)+c^2g^*\omega_{FS}.$$
		We will explain the proof in the $(z_1,z_2)$ coordinate chart, but replacing the subscript $1$ with the subscript $0$ verifies the conditions in the other chart.
		
		If $g$ is braided, $g^*(dx_1\wedge dy_1)\geq 0$. (This is a rephrasing of the condition that at each point $g$ is either positively tangent to the fibers of $\pi$ or $d(\pi\circ g)=0$.)
		
		Part (1) then follows at points where $g$ is an immersion because $g^*(dx_1\wedge dy_1)\geq 0$ and $g^*\omega_{FS}>0$ implies $(\Phi_c\circ g)^*(\omega_{FS})>0$. At cusp singularities where $g^*\omega_{FS}=0$, this argument does not apply. However, if $p\in S$ is a cusp point of $Q$, the limiting tangent space to $\Phi_c(Q)$ at $\Phi_c(g(p))$ is $d\Phi_c(T)$ where $T$ is the limiting tangent space to $Q$ at $g(p)$. Then 
		$$\omega_{FS}|_{d\Phi_c(T)} = \Phi_c^*(\omega_{FS})|_T = ((1-c^2)dx_1\wedge dy_1 +c^2\omega_{FS})|_T$$
		This is positive because  $(dx_1\wedge dy_1)|_T>0$ (since the cusp is positively transverse to the fibers of $\pi$) and $\omega_{FS}|_T>0$ (since $Q$ was symplectically braided).
		
		The hypothesis in part (2) that $g$ is positively transverse to the fibers of $\pi$ is equivalent to saying that $g^*(dx_1\wedge dy_1)>0$.
		Let $A$ denote a fixed area form on $S'$. By compactness and the fact that $dx_1\wedge dy_1>0$ on the tangent spaces to $g(S')$, we can choose $\varepsilon>0$ and $K>0$ such that $g^*(dx_1\wedge dy_1)\geq \varepsilon A$ and $g^*(\omega_{FS}) \geq -K A$. Choosing $c$ small enough so that $c^2< \frac{\varepsilon}{\varepsilon+K}$ implies that
		$$(\Phi_c\circ g)^*(\omega_{FS}) \geq ((1-c^2)\varepsilon-Kc^2)A >0$$
		so $\Phi_c\circ g(S')$ is symplectic.
	\end{proof}
	
	In addition to the isotopies above which fix each fiber of $\pi$, we will need an isotopy which sends fibers to fibers but moves the critical values of the projection of the surface into a controlled region with respect to the trisection.
	
	\begin{lemma}\label{l:equator}
		If $Q$ is a symplectic braided surface, then it is isotopic through symplectic braided surfaces to a symplectic braided surface such that every critical value of $\pi\circ F$ lies on the equator of $\cpone$, $|z_1|=|z_0|$.
	\end{lemma}
	
	\begin{proof}
		Let $c_1,\cdots, c_n \in \cpone$ be the critical values of the projection $\pi|_Q\colon Q \to \cpone$. Choose small disjoint neighborhoods $U_j$ of $c_j$. Let $\psi_t\colon \cpone\to \cpone$ be an isotopy such that $\psi_t|_{U_j}$ is a projective unitary transformation and $\psi_1(c_j)$ lies on the equator of $\cpone$. Let 
		$$\Psi_t\colon \cptwo \setminus \{[0:0:1]\} \to \cptwo\setminus \{[0:0:1]\}$$ 
		be defined by $\Psi_t([z_0:z_1:z_2])=[\psi_t([z_0:z_1]):z_2]$. Then $\pi\circ \Psi_t = \psi_t\circ \pi$. Therefore the critical values of the restriction of $\pi$ to $\Psi_1(Q)$ all lie on the equator of $\cpone$. Also, $\Psi_t(Q)$ is a braided surface for all $t$. Finally, $\Psi_t$ is a symplectomorphism on the neighborhoods $\pi^{-1}(U_j)$ of the critical fibers because it is a unitary transformation there.
		
		It is possible that $\Psi_t(Q)$ will fail to be symplectic for some $t>0$, but we do know that $\Psi_t(\pi^{-1}(U_j)\cap Q)$ will be symplectic for each $j$. Let $Q_j=Q\cap \pi^{-1}(U_j)$ and let $Q' = Q\setminus (\cup_j Q_j)$. Since $\Psi_t$ is fiber preserving, $\Psi_t(Q')$ is a compact surface which is everywhere transverse to the fibers of $\pi$. Then by Lemma~\ref{l:rescale}(2), there exists a sufficiently small value of $c_t>0$ such that $\Phi_{c_t}(\Psi_t(Q'))$ is symplectic. Since $[0,1]$ is compact, we can choose a uniform $0<c\leq 1$ such that $\Phi_c(\Psi_t(Q'))$ is symplectic for all $t\in[0,1]$. By Lemma~\ref{l:rescale}(1), $\Phi_c(\Psi_t(Q_j))$ is also symplectic because $\Psi_t(Q_j)$ is symplectic. Therefore $\Phi_c\circ \Psi_t$ is a symplectic isotopy from $\Phi_c(Q)$ to $\Phi_c\circ \Psi_1(Q)$. Additionally, $\Phi_{1+t(c-1)}(Q)$ is a family of symplectic braided surfaces for $t\in[0,1]$, interpolating between $Q$ and $\Phi_c(Q)$. Concatenating these yields the required isotopy from $Q$ to $\Phi_c(\Psi_1(Q))$.
	\end{proof}
	
	Recall the genus one Weinstein trisection $\TT_0$ for $\cptwo$ described in Section~\ref{sec:Intro} and shown schematically in the moment map image of the standard toric action in Figure~\ref{fig:toric}.
	In affine coordinates $(z_1,z_2)$, the preimage of equator of $\cpone$ is $\{|z_1|=1\}$. Thus, the handlebody
	$H_2=\{|z_1|=1,|z_2|\leq1\}$ is contained in the preimage of the equator.
	
	\begin{lemma}\label{l:bridgepos}
		Let $Q=g(S)$ for $g\colon S\to \cptwo$ be a symplectic braided surface, such that the critical values of $\pi\circ g$ lie on the equator of $\cpone$. Then there exist neighborhoods $\{V_j\}$ of the critical values of $\pi\circ g$ and a map $\phi\colon \cpone\to \R_{>0}$, which is constant on each $V_j$, such that the induced fiber preserving diffeomorphism $\Phi$ has $\Phi(Q)$ in singular bridge position with respect to the standard trisection on $\cptwo$.
	\end{lemma}
	
	\begin{proof}
		
		We begin by specifying the construction of $\phi$ and the induced $\Phi$ for a given braided surface $Q$. In the second part of the proof, we verify that $\Phi(Q)$ is in singular bridge position.
		
		We assume through a $C^1$-small symplectic isotopy that $Q$ is disjoint from the core of $H_2$ where $|z_1|=|z_0|$ and $|z_2|=0$ (general position). Let $A$ be an open neighborhood of the equator of $\cpone$ (where $|z_0|=|z_1|$) such that every point in $Q\cap \pi^{-1}(A)$ has $|z_2|\neq 0$. We assume $A$ is disjoint from open neighborhoods of the poles $[0:1]$ and $[1:0]$.
		
		As a reference, we will utilize two functions, $f_0:\cptwo\setminus \{z_0=0\} \to \R$ and $f_1:\cptwo\setminus\{z_1=0\}\to \R$ defined by 
		$$f_0([z_0:z_1:z_2]) = |z_2|/|z_0| \qquad \text{ and } \qquad f_1([z_0:z_1:z_2])=|z_2|/|z_1|.$$
		Let $L=\{[z_0:z_1:z_2]\in \cptwo \mid |z_1|\leq |z_0| \}$ and $R= \{[z_0:z_1:z_2]\in \cptwo \mid |z_1|\geq |z_0| \}$. Notice that $H_1= f_0^{-1}(1) \cap L$, and $H_3 = f_1^{-1}(1) \cap R$. Equivalently, the function 
		$$f=\min\{f_0,f_1\}$$ 
		agrees with $f_0$ on $L$ and with $f_1$ on $R$, so $f^{-1}(1)=\partial Z_3$. (Note that $f$ can be thought of as a well-defined continuous function on $\cptwo\setminus\{[0:0:1]\}$ because $f_0$ is well-defined on $L\setminus\{[0:0:1]\}$ and $f_1$ is well-defined on $R\setminus\{[0:0:1]\}$.)
		
		Since $Q$ is braided, it is disjoint from $[0:0:1]$. Let $\delta>0$ be a sufficiently small constant such that $f_0|_{Q\cap L} < 1/\delta$ and $f_1|_{Q\cap R} <1/\delta$. Since $Q\cap L$ and $Q\cap R$ are compact, such a $\delta$ exists. Observe that $\Phi_\delta(Q)$ is then completely contained in $Z_1\cup Z_2$. Note that if all the critical points of $Q$ lie above the equator of $\cpone$, they will all be contained in $H_2$. Away from the critical points, $Q$ intersects $H_2$ transversally, since each meridional disk of $H_2$ is contained a fiber of $\pi$. Therefore, away from the singular fibers, $Q\cap H_2$ is a braid, transverse to the meridional disks.
		
		Now let $U_1,\ldots, U_n\subset A$ be disjoint regular open neighborhoods in $\cpone$ of the critical values $v_1,\ldots, v_n$ of $\pi\circ g$. Note that it is possible to choose these neighborhoods inside of $A$ because the critical values are assumed to occur along the equator. 
		
		Because $U_i\subset A$, when we are focusing on $\pi^{-1}(U_i)$, we can work in the affine chart where $z_0=1$ with affine coordinates $(z_1,z_2)$. Note that in these coordinates $\pi(z_1,z_2)=z_1$, so $z_1$ is a parameter on the $U_i$. We will use polar coordinates where $z_1=r_1e^{i\theta_1}$ and $z_2=r_2e^{i\theta_2}$. By our assumption on $A$, the points on $Q\cap \pi^{-1}(U_i)$, have $r_1\neq 0$ and $r_2\neq 0$. Therefore, $\theta_1$ and $\theta_2$ are well-defined coordinates at all points on $Q\cap \pi^{-1}(U_i)$.
		We take a moment now to understand how these coordinates interact with the picture of $Q$, using the fact that $Q$ is braided, and will potentially shrink our choice of $U_i$ using this information. 
		
		\begin{figure}
			\centering
			\includegraphics[scale=.7]{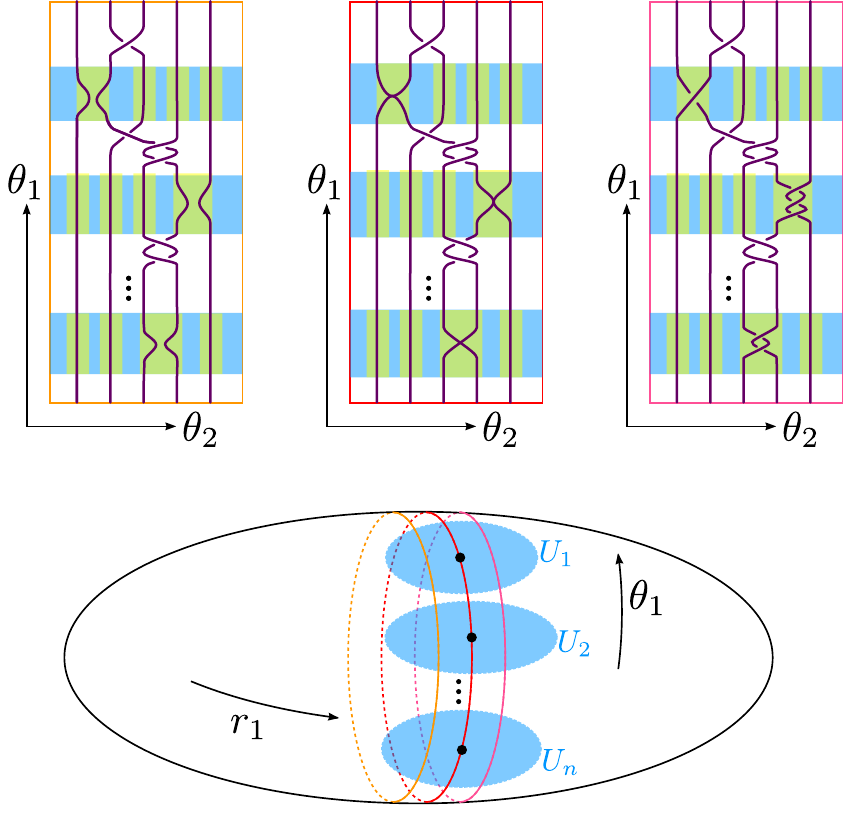}
			\caption{On the bottom we see $\cpone$ parameterized by coordinates $r_1,\theta_1$, containing the neighborhoods $U_i$ of the critical values $v_i$. On the top we look at three slices where $r_1=1-\varepsilon$, $r_1=1$ and $r_1=1+\varepsilon$. The braided surface $Q$ intersects these slices in a braid (with singularities when $r_1=1$). The $r_2$ coordinate is projected out to obtain the braid diagram (but crossings data is included to indicate relative heights), and opposite sides of each rectangle are identified to form a torus. The preimage of each $U_i$ appears as a blue band in each of these slices. The yellow highlighted rectangles inside the $U_i$ bands correspond to points in $B_{k}$ where $\theta_2\in (a_{i,k},b_{i,k})$ and $r_1e^{i\theta_1}\in U_i$. }
			\label{fig:braidcoords}
		\end{figure}
		
		Consider Figure~\ref{fig:braidcoords}. If we fix a value of $r_1$, the intersection of $Q$ with this slice where $r_1=c_r$ is a braid (transverse to the $\theta_1=c_\theta$ slices), except at the critical points which all lie in the slice $r_1=1$. We show these slices in Figure~\ref{fig:braidcoords} where we have projected out the $r_2$ coordinate. The intersection of $\pi^{-1}(U_i)$ with these slices appear as annular bands in the $(\theta_1,\theta_2)$ tori.

		By our choice of $U_i$, $\pi^{-1}(U_i)\cap Q$ includes at least one connected component which is a neighborhood of a cusp, node, or tangency in $Q$, together with other disjoint smooth disks corresponding to the other strands of the braid not involved with the critical point. Focusing in on a neighborhood of the cusp, node, or tangency, we can see this portion of $Q$ as a movie traced out by frames $\{r_1=c_r\}$ for $c_r\in (1-\eta, 1+\eta)$. 
		For each type of singularity, there is a different local movie of braids (where $c_r$ is the time direction of the movie) that becomes singular in $H_2 \subset \{r_1=1\}$.		
		At a tangency, a positive half-twist (single crossing) is added to the braid after passing through the tangency point -- Figure \ref{fig:tangency}. At a positive node, a full positive twist is added between two strands of the braid -- Figure \ref{fig:node}. Similarly, a negative node introduces a full negative twist. At a cusp, three positive half-twists are added to the braid after passing through the cusp -- Figure \ref{fig:cusp}. These all follow from calculating the braid monodromy directly from the complex local models for these three singularity types.
		
		Now, keeping Figure~\ref{fig:braidcoords} in mind, we will shrink the $U_i$, such that it is possible to choose disjoint intervals $(a_{i,k},b_{i,k})$, such that 
		$$\{ (z_1,z_2)\in Q \mid z_1\in U_i,\; \theta_2\in (a_{i,k},b_{i,k})\}$$
		contains a single connected component of $Q\cap \pi^{-1}(U_i)$ and is disjoint from the other components. In particular, the $U_i$ must be sufficiently small so that the $U_i$ bands in the constant $r_1$ slices avoid any braid crossings which are unrelated to the singularities.
		Note that to achieve this, we must assume that $Q$ is in generic position, which can be achieved by a $C^1$-small symplectic isotopy of $Q$ at the beginning.
		
				\begin{figure}
					\centering
					\includegraphics[scale=.3]{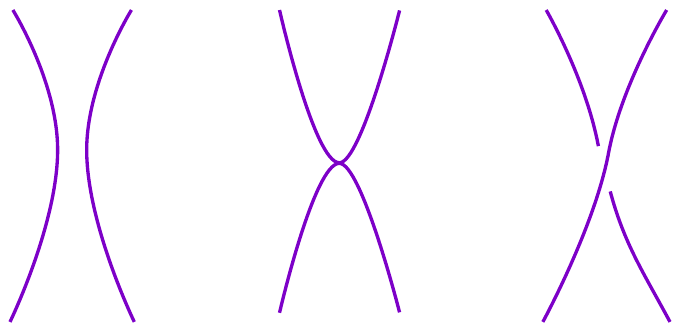}
					\caption{The movie showing a braided surface near a tangency.}
					\label{fig:tangency}
				\end{figure}
				
				\begin{figure}
					\centering
					\includegraphics[scale=.3]{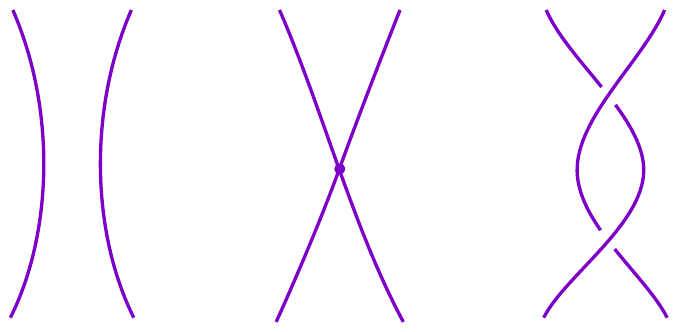}
					\caption{The movie showing a braided surface near a positive node.}
					\label{fig:node}
				\end{figure}
				
				\begin{figure}
					\centering
					\includegraphics[scale=.3]{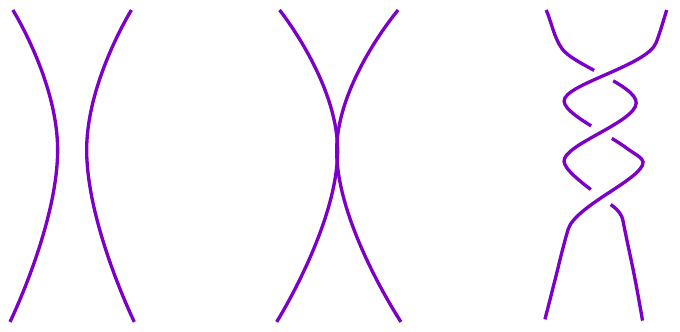}
					\caption{The movie showing a braided surface near a cusp.}
					\label{fig:cusp}
				\end{figure}

		Let $V_i$ be a regular open neighborhood of the critical value of $\pi\circ g$, $v_i$ whose closure is a disk contained in $U_i$. 
		We adjust these neighborhoods such that the functions $f_j|_{\pi^{-1}(U_i\setminus V_i)\cap Q}$, for $j=0,1$ have no critical points (again, we assume $Q$ is in generic position to ensure such critical points are isolated from the beginning by a $C^1$ small symplectic isotopy of $Q$). This concludes the adjustments of the neighborhoods $U_i$ and $V_i$.		
		
		Next, for each $i=1,\dots, n$, choose constants $K_i>1$ such that 
		$$f|_{\pi^{-1}(\overline{V_i})\cap Q}>\frac{1}{\delta K_i}.$$
		Let $\phi:\cpone\to \R_{>0}$ be a smooth function which is identically $\delta K_i$ on $\overline{V}_i$ and identically $\delta$ outside of $\cup_{i=1}^n U_i$, with no critical points in $U_i\setminus \overline{V}_i$. Let $\Phi:\cptwo\setminus\{[0:0:1]\}\to \cptwo\setminus\{[0:0:1]\}$ be the fiber preserving map induced by $\phi$ as above.
		
		We will next look at $\Phi(Q)$ and check it is in singular bridge position with respect to the standard trisection of $\cptwo$. Note that $\Phi(Q)$ is isotopic to $Q$ through a family of singular surfaces $\Phi_t(Q)$ where $\Phi_t$ is the fiber preserving map induced by $\phi_t=t\phi+1-t$.
		
		We have illustrated the effect of the fiber preserving map $\Phi$ on $Q$ near a tangency in Figure~\ref{fig:isotopetangency}. Note that the fiber preserving map $\Phi$ allows us to modify the $z_2$ radial coordinate; this direction is depicted as vertical within each frame, while the coordinate $r$ varies as you move from one box to another horizontally. The top line shows $Q$ before applying $\Phi$, while the bottom line shows $\Phi(Q)$.
		
		\begin{figure}
			\centering
			\includegraphics[width=\textwidth]{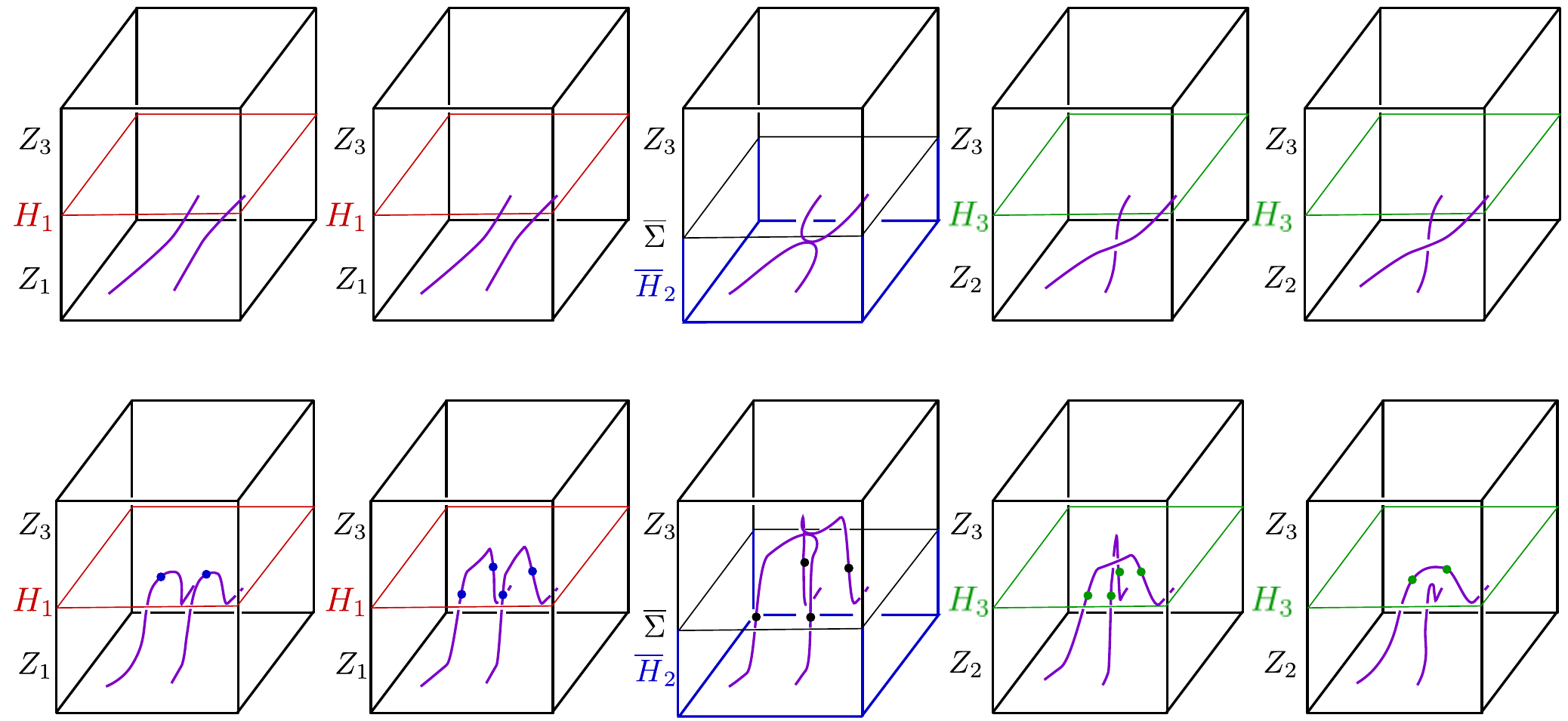}
			\caption{A fiber preserving isotopy of a neighborhood of a tangency point of $Q$ with the fibers of the pencil $\pi$. The top row shows the local portion of the surface at the beginning of the isotopy, and the bottom row shows the local portion of the surface at the end of the isotopy.}
			\label{fig:isotopetangency}
		\end{figure}
		
		Replacing the braid movie in Figure~\ref{fig:isotopetangency} with the braid movies for the node and cusp models (Figures~\ref{fig:node} and~\ref{fig:cusp}) would give the corresponding movie depicting $\Phi(Q)$ near each of these singularities.
		
		To verify that $\Phi(Q)$ is in singular bridge position with respect to the standard trisection of $\cptwo$, we first check that $\Phi(Q)\cap Z_\lambda$ is a singular disk tangle in $Z_\lambda$ for $\lambda \in\{1,2,3\}$. Then we will check that $\Phi(Q)\cap H_\lambda$ is a trivial tangle in $H_\lambda$ for $\lambda\in\{1,2,3\}$.
		
		First we look at $\Phi(Q)\cap Z_3$. We need to show that this forms a singular disk tangle. We will enclose each connected component of $\Phi(Q)\cap Z_3$ in a ball in $Z_3$ which has half of its boundary on $\partial Z_3$. Notice that by the choice of $\delta$, $\Phi(Q)\cap Z_3\subset \pi^{-1}(\cup_{i=1}^n U_i)$, and by choice of $K_i$, $\Phi(Q)\cap \pi^{-1}(\overline{V}_i)\subset Z_3$. Furthermore, $f(z_1,z_3)< K_i$ for $(z_1,z_2)\in \Phi(Q)\cap \pi^{-1}(U_i)$ because $\phi\leq \delta K_i$ on $\pi^{-1}(U_i)$ and $f<\delta$ on $Q$. Let 
		$$B_{i,k} = \{ (z_1,z_2=r_2e^{i\theta_2}) \mid z_1\in U_i, \, \theta_2\in (a_{i,k},b_{i,k}), \, r_2\leq K_i, \, f(z_1,z_2)\geq 1 \}.$$
		Then $B_{i,k}$ is a $4$-ball contained in $Z_3$ (with half of its boundary on $\partial Z_3$), which contains a single component of $\Phi(Q)\cap Z_3$. Furthermore, $\Phi(Q)\cap Z_3$ is contained in the union of all the $B_{i,k}$. We will verify that for each $(i,k)$, $\Phi(Q)\cap B_{i,k}$ is either a boundary parallel disk with unknotted boundary or a singular disk which is the cone on a trefoil or a Hopf link in the boundary. Because $\phi$ is constant (equal to $\delta K_i$) on $\overline{V}_i$, the braided structure of $Q$ in $\cptwo$ determines the topology of the disk and its boundary link for $\Phi(Q) \cap \pi^{-1}(\overline{V}_i)\subset B_{i,k_i}$. By our assumptions on the braided structure, for one $k_i$, there is a critical point of $\pi$ which is either a tangency, a node, or a cusp.  In the case of a tangency, $\Phi(Q)\cap \pi^{-1}(\overline{V}_i)\cap B_{i,k_i}$ is a boundary parallel disk with unknotted boundary, as can be demonstrated by the $3$-ball shown in slices in Figure~\ref{fig:tangency} with half its boundary on $\Phi(Q)$ and the other half a disk in $\partial( \overline{V}_i\cap B_{i,k_i})$. In the cusp case, $\Phi(Q)\cap \pi^{-1}(\overline{V}_i)\cap B_{i,k_i}$ is a cone on the trefoil and in the node case it is the cone on a Hopf link. For other values of $k\neq k_i$, $\Phi(Q) \cap \pi^{-1}(\overline{V}_i)\subset B_{i,k}$ will just be a trivial disk with unknotted boundary (a $3$-ball with half its boundary on $\Phi(Q)$ and half a disk on the boundary of $\pi^{-1}(\overline{V}_i)\cap B_{i,k}$ can easily be constructed by taking a trivial movie of $2$-disks with the analogous property in the trivial movie of $\Phi(Q)$ in $\pi^{-1}(\overline{V}_i)\cap B_{i,k}$ where $\Phi(Q)$ is just a single unknotted arc in each frame). Because we assumed that $\phi:\cpone \to \R_{>0}$ has no critical points in $U_i\setminus \overline{V}_i$, and that $f_j|_Q$ for $j=0,1$ has no critical points in $\pi^{-1}(U_i\setminus V_i)$, the cobordism from the boundary of $\Phi(Q)\cap \pi^{-1}(\overline{V_i})\cap B_{i,k}$ to the boundary of $\Phi(Q)\cap B_{i,k}$ is trivial. Therefore $\Phi(Q)\cap Z_3$ is a singular disk tangle.
		
		Next we look at $\Phi(Q)\cap Z_1$ and $\Phi(Q)\cap Z_2$. Note that all of the critical points of $\Phi(Q)$ as a braided surface are contained in $Z_3$. This means that $\pi:Q\to \cpone$ has no critical points in $Z_1$ or $Z_2$. On $Z_1$, which is contained in the subset of $\cptwo$ where $z_0\neq 0$, we use coordinates $(z_1,z_2)$. In these coordinates $\pi(z_1,z_2)=z_1$. We can further project to the radial component $r_1=|z_1|$ by the function $\pi_{r_1}$. Then $\pi_{r_1}:\Phi(Q)\to \R$ has no critical points except where $\Phi(Q)$ intersects $\{z_1=0\}$, where is has index $0$ critical points. This almost shows that $\Phi(Q)\cap Z_1$ is a trivial disk tangle in $Z_1$, except that part of the boundary of $Z_1$ is not cut out by a level set of $r_1$, but rather a level set of $r_2$. We check that wherever $\Phi(Q)$ approaches this $H_1$ part of $\partial Z_1$, that the projection to $r_2$ has no critical points. We know that $\Phi(Q)$ can only intersect the boundary of $Z_1$ inside of the union of the $\pi^{-1}(U_i)$ since $\phi\equiv \delta$ outside the $U_i$. Because $\phi$ and $f_0|_Q$ have no critical points in $U_i\setminus \overline{V_i}$, and $\Phi(Q)\cap \pi^{-1}(U_i) \cap Z_1\subset \pi^{-1}(U_i\setminus \overline{V_i})$, we see that $f_0|_{\Phi(Q)}$ has no critical points in $Z_1\cap \pi^{-1}(U_i)$. Therefore $\Phi(Q)\cap Z_1$ is a trivial disk tangle in $Z_1$. We can make a similar argument for $\Phi(Q)\cap Z_2$, using coordinates $(z_0,z_2)$ instead of $(z_1,z_2)$ and $\pi_{r_0}$ instead of $\pi_{r_1}$, to see that $\Phi(Q)\cap Z_2$ is a trivial disk tangle in $Z_2$.		
		
		Now that we have checked we have singular disk tangles in each of the three $4$-dimensional sectors, we just need to check that the intersections of $\Phi(Q)$ with the $3$-dimensional handlebodies $H_1$, $H_2$, and $H_3$ are trivial tangles. We start with $H_2$, which lies above $\{|z_1|=1\}$. The intersection of $Q$ with $H_2$ is represented in the center diagram at the top of Figure~\ref{fig:braidcoords}. $\Phi(Q)\cap H_2$ differs from $Q\cap H_2$ by deleting the portions of the braid in $\pi^{-1}(U_i)$ which get pushed out of $H_2$ into $Z_3$ by $\Phi$. From this, we see that $\Phi(Q)\cap H_2$ is a union of braided tangles which are trivial because the braiding can be undone by an isotopy of the endpoints of the tangle in the boundary of $H_2$. Note that this assumes that there is at least one $U_i$ where $Q$ is pushed up by $\Phi$ (if there are no critical points of $\pi$ on $Q$, we can still choose a trivial piece on which to let $\Phi$ push $Q$ up into $Z_3$).
		
		Next we check that $\Phi(Q)\cap H_1$ and $\Phi(Q)\cap H_3$ are trivial tangles. To understand these tangles, we use the perspective of the movie with $\{r_1=c_r\}$ frames, and see how $\Phi(Q)$ intersects $H_1$ or $H_3$ in each frame. Note that this movie slices $H_1$ or $H_3$ into concentric tori shrinking down towards the core of the solid torus (the core is avoided by $\Phi(Q)\cap H_i$). $\Phi(Q)$ intersects each of these concentric tori in a finite collection of points which are contained in $\cup_i \pi^{-1}(U_i)$. We can track the movie traced out by these points to see the relevant tangle. We draw these tangles for each of the models (tangency, node, cusp, or trivial) in Figure~\ref{fig:tangle}.
		
		\begin{figure}
			\centering
			\includegraphics[scale=.5]{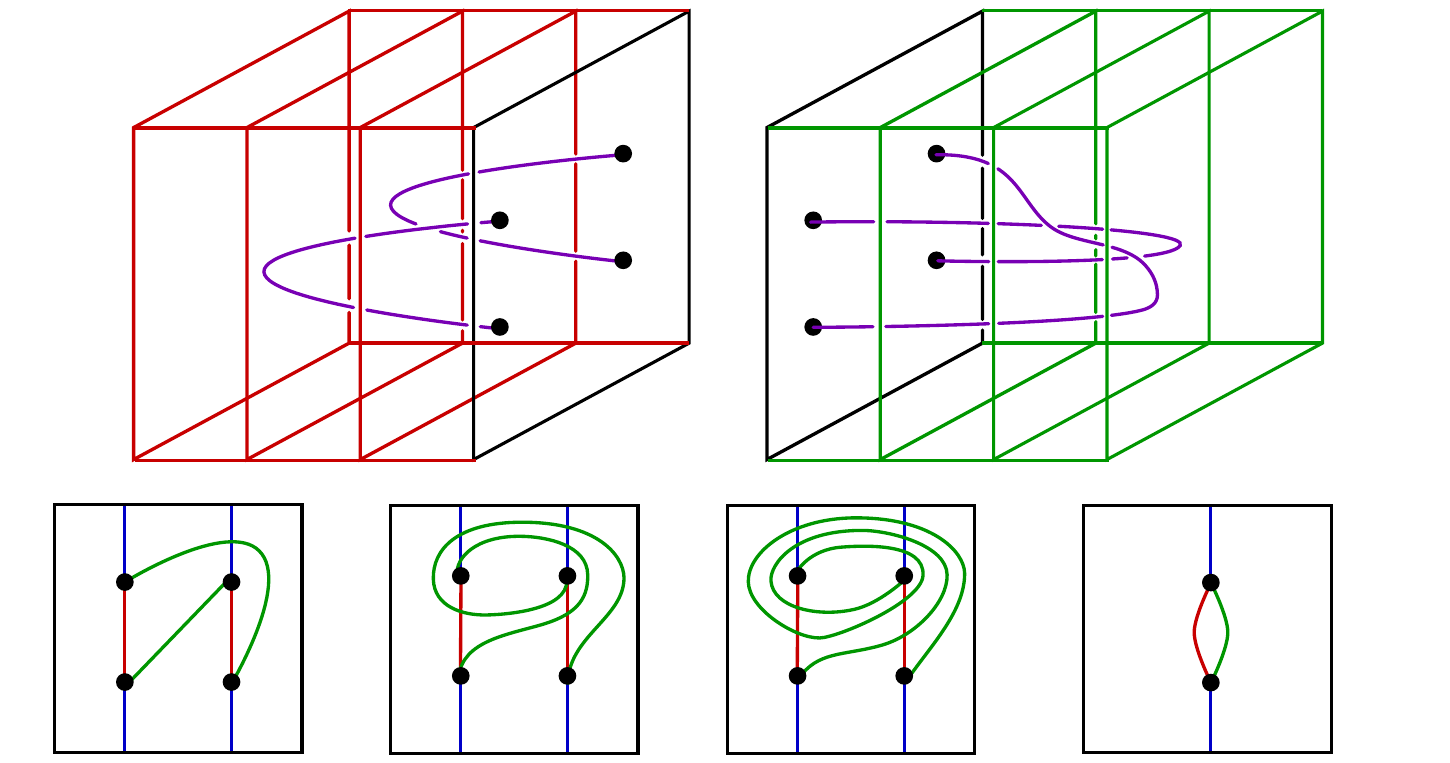}
			\caption{(Top) The pieces of the tangles in $\overline{H}_1$ and $H_3$ after isotoping $Q$ near a tangency. (Note that $\partial Z_3 = \overline{H_1}\cup_\Sigma H_3$.) Gluing these pieces together gives an unknot. (Bottom, left to right) The projections of the corresponding tangles in $\Sigma$ for tangencies, nodes, cusps, and trivial pieces respectively. Note that in this figure $\Sigma$ is the oriented boundary of $H_3$ and the oppositely oriented boundary of $\overline{H}_1$.}
			\label{fig:tangle}
		\end{figure}
		
		For a tangency, node, node or cusp, in the movie of points in the tori $H_1\cap \{r_1=c_r\}$ for $c_r\leq1$ or $H_3\cap \{r_1=c_r\}$ for $c_r\geq 1$, we see the births of two pairs of points. The track of these births in each handlebody is a pair of arcs. For a trivial model, a single pair of points is born which track out a single arc. The track of these births on $\Sigma$ is the projection of these arcs to the torus. Because arcs of the same color can be projected up to isotopy as embeddings, the tangles $Q\cap H_1$ and $Q\cap H_3$ are locally trivial. Note that for the total surface there will be many disjoint copies of these models inserted in disjoint regions of the torus.  It follows that $\Phi(Q)\cap H_1$ and $\Phi(Q)\cap H_3$ are trivial tangles.

\end{proof}
	
\textbf{Warning:} When drawing pictures and projections, one needs to be very careful to keep track of which orientation one is using on each handlebody. \emph{Our convention is that $H_\lambda$ is positively oriented as part of the boundary of $Z_\lambda$ and that $\Sigma$ is the positively oriented boundary of $H_\lambda$.} When we look at a braid in $V_r$, we think of $V_r$ as a parallel copy of $\overline{H}_2$, with boundary $\overline{\Sigma}$.  As we pass singularities, the braid obtains positive twists (when viewed in $\overline{H}_2$ which is the orientation it inherits as the boundary of $Z_1$). In particular, the natural projection of the positive braid is to $\overline{\Sigma}$. Pushing parts of this braid into $Z_3$ gradually, we trace out the intersection of $Q$ with $H_3$. We can then project this trace onto $\overline{\Sigma}$ and reflect to obtain the projection to $\Sigma$. See Figure~\ref{fig:cuspgamma} for the case of the cusp.
		
\begin{figure}
	\centering
	\includegraphics[width=\textwidth]{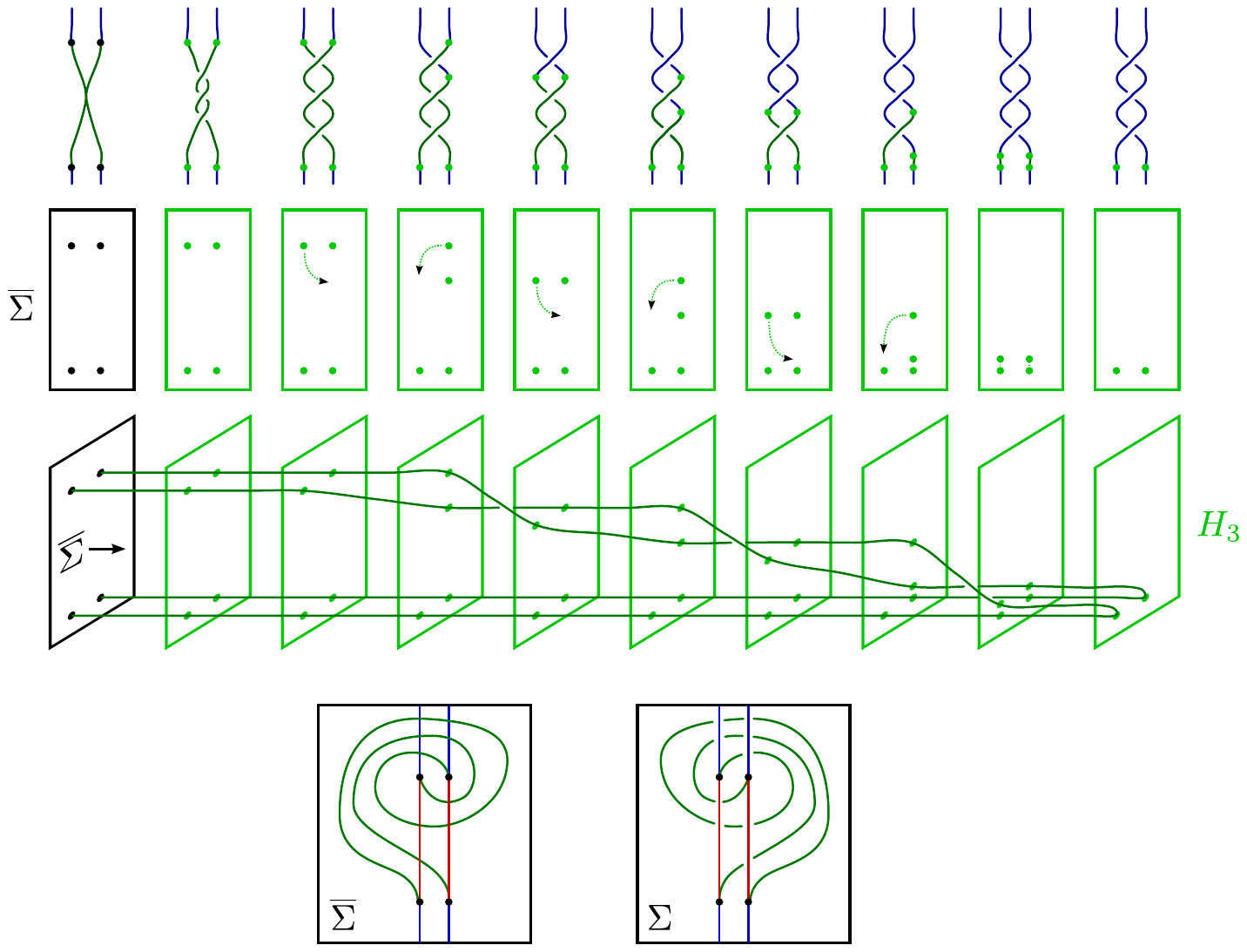}
	\caption{(Top row) Pushing the neighborhood of the cusp into $Z_3$. (Middle rows) Tracing out the intersection with $H_3$. (Bottom row) Projecting the tangle to $\overline{\Sigma}$ and its reflection~$\Sigma$.}
	\label{fig:cuspgamma}
\end{figure}

Note that when we have the projections of the arcs from $H_1$, $H_2$, and $H_3$ all on $\Sigma$, we can think of one set of arcs (say the blue arcs of $H_2$) as being isotoped to lie exactly on $\Sigma$. Pushing the red arcs slightly up and the green arcs slightly down is consistent with viewing all three sets of arcs inside a neighborhood of $\Sigma$ in $H_3\cup_\Sigma \overline{H}_1$. In particular, the knot formed by the union of the red and green arcs in the bottom-right frame of Figure~\ref{fig:cuspgamma} gives a positive trefoil with this convention. Since $H_3 \cup_\Sigma \overline{H}_1$ is the oriented boundary of $Z_3$, we see that this positive trefoil is the link of the cusp singularity in $\partial Z_3$ as expected.

	\begin{lemma}\label{l:scaletrisect}
		If $\TT_0$ is the standard trisection of $\cptwo$, then its image under the rescaling, $\Phi_c(\TT_0)$ is a genus one Weinstein trisection.
	\end{lemma}
	
	\begin{proof}
		The three sectors of the standard trisection are given in homogeneous coordinates by $[z_0:z_1:z_2]$ by
		$$Z_\lambda = \{|z_\lambda|,|z_{\lambda+1}|\leq |z_{\lambda+2}|\},$$
		for $i=0,1,2$ (where the indices are considered mod $3$). The effect of the rescaling map $\Phi_c$, in terms of homogeneous coordinates on the target $[w_0:w_1:w_2]$ is $w_0=z_0$, $w_1=z_1$, $w_2=cz_2$. Therefore the image of the trisection under the rescaling is given by the three sectors
		$$Z_1' = \{|w_1|\leq |w_0|, |w_2|\leq c|w_0| \}, \hspace{.3cm} Z_2' = \{|w_0|\leq |w_1|, |w_2|\leq c|w_1| \}, \hspace{.3cm} \textrm{ and }  \hspace{.3cm} Z_3' = \{|w_0|,|w_1|\leq \frac{1}{c}|w_2|\}$$
		
		Note that the Fubini-Study form on any affine chart is explicitly exact: $\omega_{FS} = d(-d^\C(\log(1+|w|^2)))$. The Liouville vector field $V$ dual to the primitive $\eta = -d^\C(\log(1+|w|^2))$ has the property that it is obtained from the radial vector field $V_0$ by rescaling by a positive function, $V=hV_0$. We can see this as follows. Let $f_0(w) = |w|^2$ and let $\eta_0=-d^\C f_0$. Then $\eta_0$ is the standard Liouville form for the Darboux symplectic form $\omega_0=d\eta_0$ and its vector field dual is the radial vector field $V_0 = x_1\partial_{x_1}+y_1\partial_{y_1}+x_2\partial_{x_2}+y_2\partial_{y_2}$ (with $\iota_{V_0}d\eta_0 = \eta_0$). Now, the chain rule shows that $\eta = g\eta_0$ where $g(w) = (1+|w|^2)^{-1}$ and $\omega_{FS} = d\eta = dg\wedge \eta_0+gd\eta_0$. Then if $V=hV_0$ is a multiple of the radial vector field by a function we have 
		$$\iota_V\omega_{FS} = h\iota_{V_0}(dg\wedge \eta_0+gd\eta_0) = h(dg(V_0)\eta_0+g\eta_0).$$
		Therefore by choosing $h$ appropriately so that $h(dg(V_0)+g) = g$, we get that $V=hV_0$ is the vector field dual to the Liouville form $\eta$.
		
		Each of our three rescaled sectors lies in one of the three affine charts, so it suffices to show that the corresponding Liouville vector field $V$ (or equivalently the radial vector field $V_0$) is outwardly transverse to the boundary of the sector (after smoothing the corner).
		
		In the affine chart where $w_0=1$, we have the sector $\{|w_1|\leq 1, |w_2|\leq c \}$ with boundary 
		$$\{|w_1|^2=1, |w_2|\leq c\} \cup \{|w_1|\leq 1, |w_2|^2=c^2 \}.$$
		The radial vector field $V_0$ is transverse to the first piece because $d(|w_1|^2)(V_0) = |w_1|^2=1>0$ and on the second piece because $d(|w_2|^2)(V_0) = |w_2|^2=c>0$. We can find an approximation smoothing the corner by looking at the hypersurface
		$$Y=\left\{ \frac{1}{N}\left( |w_1|^2+\frac{1}{c^2}|w_2|^2 \right) + |w_1|^{2N}+\frac{1}{c^{2N}}|w_2|^{2N} =1  \right\}$$
		The vector field $V_0$ is positively transverse to $Y$ because
		$$d\left(\frac{1}{N}\left( |w_1|^2+\frac{1}{c^2}|w_2|^2 \right) + |w_1|^{2N}+\frac{1}{c^{2N}}|w_2|^{2N}  \right)(V_0)$$
		$$= \frac{1}{N}d(|w_1|^2)(V_0)+\frac{1}{Nc^2}d(|w_2|^2)(V_0)+N|w_1|^{2N-2}d(|w_1|^2)(V_0)+\frac{N}{c^{2N}}|w_2|^{2N-2} d(|w_2|^2)(V_0)$$
		which is a positive combination of positive quantities (when $w_1=0$ or $w_2=0$ some of the quantities are $0$ and others are positive, otherwise all terms in the sum are positive).
		
		In the other affine charts, the computations are similar to show that the radial vector field $V_0$, and thus the Liouville vector field $V$, is positively transverse to the (smoothed) boundary of the sectors.
	\end{proof}
	
	Now we have all the tools to deduce the main result of this section.
	
	\begin{theorem}\label{thm:symplbridgepos}
		If $g\colon S\to \cptwo$ is a symplectic braided surface in $\cptwo$, then there is an ambient symplectic isotopy $H_t\colon \cptwo\to \cptwo$ such that $H_1(g(S))$ is in bridge position with respect to a genus one Weinstein trisection of $\cptwo$.
	\end{theorem}
	
	\begin{proof}
		By Lemma~\ref{l:equator}, we get a symplectic isotopy $\Psi_t\colon \cptwo\to \cptwo$ such that $\Psi_0=id$ and the critical values of $\pi\circ \Psi_1\circ g$ are points $r_1,\ldots, r_n$ on the equator of $\cpone$.
		
		Then by Lemma \ref{l:bridgepos}, there exists a function $\phi\colon \cpone\to \R_{>0}$ such that 
		\begin{itemize}
			\item $\phi$ is constant on open neighborhoods $V_j$ of $r_j$ for $j=1,\cdots, n$, and
			\item for the map $\Phi$ induced by $\phi$, $\Phi(Q)$ is in bridge position with respect to the standard trisection on $\cptwo$.
		\end{itemize}
		
		Let 
		$$C=\max_{[z_0:z_1]\in \cpone}(\phi([z_0:z_1])).$$
		Then $\Phi_{1/C}\circ \Phi(Q)$ is in bridge position with respect to $\Phi_{1/C}(\TT_0)$.
		Let $\phi_t = \frac{t}{C}\phi+1-t$. Then for $t\in[0,1]$, $\phi_t\colon \cpone\to \R_{>0}$ induces a fiber preserving map $\Phi_t\colon\cptwo\to \cptwo$. Note that $\phi_t$ is constant and less than or equal to $1$ in $V_j$, for $j=1,\ldots, n$. Since $\Psi_1(g(S))$ is symplectic, Lemma~\ref{l:rescale}(1) applied to the restriction of $g$ to these neighborhoods, implies that $\Phi_t(\Psi_1(g(S)))\cap \pi^{-1}(V_j)$ is symplectic for all $t\in[0,1]$ and $j=1,\ldots, n$.
		
		Outside of $\pi^{-1}(V_j)$ it is possible that $\Phi_t(\Psi_1(g(S)))$ fails to be symplectic at some time $t$. Let $S' = S\setminus (\pi\circ \Psi_1\circ g)^{-1}(V_j)$. Note that $\pi=\pi\circ\Phi_t$. Because the critical values of $\pi\circ \Phi_t \circ \Psi_1 \circ g$ are contained in the neighborhoods $V_j$, the map $\Phi_t\circ \Psi_1\circ g\colon S'\to \cptwo$ is positively transverse to the fibers of $\pi$ and $S'$ is compact. Therefore by Lemma~\ref{l:rescale}(2), there exists a sufficiently small $0<c_t\leq 1$ such that the induced constant rescaling map $\Phi_{c_t}$ has $\Phi_{c_t}(\Phi_t(\Psi_1(g(S'))))$ is symplectic. Since $[0,1]$ is compact, we can choose $c=\min_{t\in [0,1]}\{c_t\}>0$ such that $\Phi_c(\Phi_t(\Psi_1(g(S'))))$ is symplectic for all $t\in[0,1]$. Then $\Phi_c(\Phi_t(\Psi_1(g(S))))$ is symplectic for all $t\in[0,1]$ because it is symplectic outside $\pi^{-1}(V_j)$ by the choice of $c$ and it is symplectic inside $\pi^{-1}(V_j)$ because $\Phi_t(\Psi_1(g(S)))$ was symplectic and $\Phi_c$ preserves symplecticness here by Lemma~\ref{l:rescale}(1).
		
		Using all this information, we can construct the isotopy through symplectic braided surfaces from $g\colon S\to \cptwo$ to a symplectic braided surface in bridge position with respect to a genus one Weinstein trisection on $\cptwo$ as the concatenation of the following isotopies. First, we isotope $g(S)$ to $\Psi_1(g(S))$ using $\Psi_t$ from Lemma~\ref{l:equator}. Then we isotope $\Psi_1(g(S))$ to $\Phi_c(\Psi_1(g(S)))$ through the isotopy of constant rescaling maps $\Phi_{1-t(1-c)}$; this is an isotopy through symplectic braided surfaces by Lemma~\ref{l:rescale}(1). Finally, we concatenate with the isotopy $\Phi_c\circ \Phi_t\circ \Psi_1\circ g\colon S\to \cptwo$ described above. Because $\Phi_1$ has the property that $\Phi_1\circ\Psi_1\circ g(S)$ is in bridge position with respect to the trisection $\Phi_{1/C}(\TT_0)$ on $\cptwo$, we have that $\Phi_c\circ\Phi_1\circ\Psi_1\circ g(S)$ is in bridge position with respect to the image $\Phi_{c/C}(\TT_0)$ of the standard trisection under $\Phi_{c/C}$. This is a genus one Weinstein trisection of $\cptwo$ by Lemma~\ref{l:scaletrisect}.
		
		By Lemma \ref{l:sympliso}, there is an ambient symplectic isotopy $H_t\colon \cptwo\to \cptwo$ such that $H_t(g(S))$ agrees with this family of symplectic braided surfaces, outside of an arbitrarily small neighborhood of the singular points. In the neighborhood of the singular points, $H_t(g(S))$ will have the same topological type of the singularity (node/cusp) but the analytic type may slightly differ. By choosing the neighborhoods of the singularities sufficiently small, we can assume that these neighborhoods are contained on the interior of $Z_3':=\Phi_{c/C}(Z_3)$ and therefore $H_1(g(S))$ is also in bridge position.
		
	\end{proof}

	\section{Lifting the Weinstein structure}
	\label{sec:lifting}

	Given a singular symplectic branched covering $f\colon (X,\omega)\to (\cptwo,\omega_{FS})$, Auroux constructs a symplectic structure $\omega_f$ on $X$ which is an arbitrarily small perturbation of $f^*\omega_{FS}$ \cite[Proposition 10]{Auroux}. Note that because $f$ has critical points, $f^*\omega_{FS}$ cannot be non-degenerate along the critical locus. To remedy this, $\omega_f$ is built by adding a small multiple of an exact form which is supported near the critical locus of $f$. Although $\omega_f$ is not identical to our original symplectic form $\omega$, we will see in Lemma \ref{l:symplectomorphic} that $(X,k\omega)$ and $(X,\omega_f)$ are symplectomorphic and this suffices to show that a Weinstein trisection of $(X,\omega_f)$ induces a Weinstein trisection of $(X,\omega)$.
	
%
%
%
	
	\begin{theorem}
			\label{thm:pullbackWeiTri}
			Suppose that $Q\subset\CP^2$ is a symplectic singular surface in bridge position with respect to a Weinstein trisection $\TT_0$ of $\CP^2$.  Let $f\colon X\to \CP^2$ be a singular branched covering with the same smooth local models as a symplectic singular branched covering (local diffeomorphism, cyclic branching, and the cusp).  Then there exists a symplectic form $\omega_f=f^*\omega_{FS}+\varepsilon d\phi$ on $X$ such that $f^*\TT_0$ is a Weinstein trisection for $(X,\omega_f)$.
	\end{theorem}
		
	\begin{proof}
		We know from Theorem~\ref{prop:branch-tri} that $f^*\TT_0$ is smoothly a trisection. We will verify that each sector with its restricted symplectic form is a Weinstein domain filling $\#^k(S^1\times S^2)$.
		
		Branched coverings have been studied in both symplectic and contact topology. The elemental question of how to lift the symplectic or contact structures from the base to the cover is answered in \cite[Proposition 10]{Auroux} and \cite[Theorem 7.5.4]{Geiges} in the symplectic and contact cases respectively. In our case, we want to lift the symplectic filling structure of $(S^3,\xi_{std})$ on $(Z_\lambda,\omega_{FS}|_{Z_\lambda})$ to a symplectic filling structure on $(\widetilde Z_\lambda,\omega_f)$ of $\#^{k_\lambda}(S^1\times S^2)$ with its unique tight contact structure. The key point we will see is that the constructions of symplectic and contact structures on the branch cover are compatible with each other. This is what will show us that the symplectic structure constructed on the cover is compatible with the pull-back trisection $f^*\TT_0$.
		
		Because the branch locus in $\cptwo$ is in bridge trisected position, its intersection with $\partial Z_\lambda$ is a smooth link. The singular points of the branch locus are away from $\partial Z_\lambda$. Let $i\colon\partial Z_\lambda\to Z_\lambda$ be the inclusion. Let $\omega_\lambda:=\omega_{FS}|_{Z_\lambda}$, and let $\eta_\lambda$ be the Liouville form ($\eta_\lambda = \omega_\lambda(V_\lambda, \cdot)$) which satisfies $d\eta_\lambda = \omega_\lambda$. The transversality of the Liouville vector field implies that $\alpha_\lambda:=i^*\eta_\lambda$ is a contact form on $\partial Z_\lambda$ (see the end of this proof for how to understand this equivalence if you are not already familiar). Because the branch locus is symplectic, this link $Q\cap \partial Z_\lambda$ is a transverse link with respect to the contact planes $\xi_\lambda = \ker\alpha_\lambda$. 
		
		The symplectic form $\omega_f$ on the cover $X$ is constructed as follows. Initially we start with the form $f^*\omega_{FS}$. This is certainly closed since $d(f^*\omega_{FS})=f^*d\omega_{FS}=0$, but it fails to be non-degenerate near the singular points of $f$. At points where $f$ is a local diffeomorphism, $f^*\omega_{FS}$ will be non-degenerate, but at singular points where $f$ is modeled on a cyclic branched covering $f(z_1,z_2)=(z_1^d,z_2)$ or a cusp $f(z_1,z_2) = (z_1^3-z_1z_2,z_2)$, $f^*\omega_{FS}$ has kernel in the $\{dz_2=0\}$ direction. Therefore, we need to add something to $f^*\omega_{FS}$ to achieve non-degeneracy. Choose coordinate charts where the local models for $f$ hold, such that the coordinate charts containing a cusp model are disjoint from $\partial \widetilde Z_\lambda$ for all $\lambda\in\Z_3$. The ramification locus $R$ in each coordinate chart is the set $\{z_1=0\}$ which is locally parameterized by the coordinate $z_2$. 
		
		Following Auroux \cite[Proposition 10]{Auroux}, choose a radius $\rho>0$ such that the polydisk $B_{2\rho}(0)\times B_{2\rho^2}(0)$ is contained in the local coordinate patches with cusp models, and the polydisk $B_{2\rho}(0)\times B_{2\rho}(0)$ is contained in the local coordinate patches in cyclic branch covering models. Let $\{U_j\}$ be an cover of the ramification locus $R$ by such polydisks.
		Let $\chi_1^j$ be a cut-off function on $\C$ that is $1$ on $B_\rho(0)\subset \C$ and $0$ outside $B_{2\rho}(0)$. Let $\chi_2^j$ be a cut-off function on $\C$, is $1$ on $B_{\rho^2}(0)\subset \C$ and $0$ outside $B_{2\rho^2}(0)$. Choose these cut-offs such that, using the coordinate models in the $U_j$, we have that $\{\chi_{s_j}^j(z_2)\}$ (where $s_j=1$ for a coordinate patch with a cyclic branching chart and $s_j=2$ for a coordinate patch with a cusp model) is a partition of unity on $R$ subordinate to $\{U_j\cap R\}$.
		
		On each $U_j$ with coordinates $z_1=x_1+iy_1$ and $z_2=x_2+iy_2$, define 
		$$\phi_j = \frac{1}{2}\chi_1^j(z_1)\chi_{s_j}^j(z_2)(x_1dy_1-y_1dx_1)$$
		where $s_j=1$ for a coordinate patch with a cyclic branching chart and $s_j=2$ for a coordinate patch with a cusp model. (Note that Auroux uses $x_1dy_1$ instead of $\frac{1}{2}(x_1dy_1-y_1dx_1)$ as a primitive for the $dx_1\wedge dy_1$, but this choice does not affect the computation.) Where the cut-off functions are identically~$1$, $d\phi_j = dx_1\wedge dy_1$. This is positive on the kernel of $df$, which makes up for the non-degeneracy of $f^*\omega_{FS}$. Finally, observe that because of the cut-off functions, each $\phi_j$ extends by zero to the entire manifold. Let $\phi=\sum_j \phi_j$.
		
		
		Finally define $\omega_f := f^*\omega_{FS}+\varepsilon d\phi$. Then for $\varepsilon$ sufficiently small, Auroux verifies that $\omega_f$ is non-degenerate \cite[Proposition 10]{Auroux}. Note that Auroux only considers cyclic branched covering models $(z_1,z_2)\mapsto (z_1^d,z_2)$ when $d=2$, but the computation is unaffected by using higher values of $d$ because the kernel of $df$ is the same and $\phi$ is constructed using the coordinates upstairs.
		
		Now restrict $\omega_f$ to $\widetilde Z_\lambda$. There,
		$$\omega_f = f^*\omega_\lambda+\varepsilon d\phi|_{\widetilde Z_\lambda} = f^*d\eta_\lambda +\varepsilon d\phi=d\left(f^*\eta_\lambda+\varepsilon\phi \right)$$
		Let $i_\lambda \colon\partial \widetilde Z_\lambda \to \widetilde Z_\lambda$ be the inclusion. Restricting this primitive to the boundary gives an induced $1$--form on $\partial \widetilde Z_\lambda$, $i_\lambda^*(f^*\eta_\lambda+\varepsilon\phi) = f^*\alpha_\lambda+\varepsilon i_\lambda^*\phi$, which we can verify is contact.  Since the ramification locus intersects $\partial \widetilde Z_\lambda$ transversally along a link, $\varepsilon i_\lambda^*\phi$ is supported in a neighborhood of this ramification link. The coordinate $z_1$ in each $U_j$ provides a coordinate on the normal bundle to this ramification link. The natural contact form on the branched cover constructed in \cite[Theorem 7.5.4]{Geiges} is (up to a constant factor) $f^*\alpha_\lambda+g(r) r^2d\theta$ where $(r,\theta)$ are polar coordinates on the normal bundle $(z_1=re^{i\theta})$ and $g(r)$ is a cut-off function. Noting that $(x_1dy_1-y_1dx_1) = r^2d\theta$, we find that the same computation verifies that $i_\lambda^*(f^*\eta_\lambda+\varepsilon\phi)$ is a contact form on $\partial \widetilde Z_\lambda$.
		
		Note that because the primitive $\widetilde \eta_\lambda:=f^*\eta_\lambda+\varepsilon\phi$ for $\omega_f|_{\widetilde Z_\lambda}$ restricts to a positive contact form on the boundary, it gives a Liouville vector field $\widetilde V_\lambda$ (defined by $\omega_f(\widetilde V_\lambda,\cdot) = \widetilde \eta_\lambda$) which is outwardly transverse to the boundary defined by $\iota_{\widetilde V_\lambda}(\omega_f|_{\widetilde Z_\lambda}) = \widetilde \eta_\lambda$. To understand why, let $i_\lambda \colon\partial \widetilde Z_\lambda \to \widetilde Z_\lambda$ be the inclusion and $\widetilde \alpha_\lambda=i_\lambda^*\widetilde \eta_\lambda$. The claim is that the property that $\widetilde V_\lambda$ be outwardly transverse to the boundary is equivalent to the condition that $\widetilde \alpha_\lambda \wedge d\widetilde \alpha_\lambda$ is a positive volume form. This is because $\widetilde \alpha_\lambda \wedge d\widetilde \alpha_\lambda = i^*\widetilde \eta_\lambda \wedge i^*\omega_f = i^*(\iota_{\widetilde V_\lambda}\omega_f \wedge \omega_f)$. Since $\omega_f\wedge \omega_f$ is a positive volume form, $\widetilde V_\lambda$ is outwardly transverse to $\partial \widetilde Z_\lambda$ if and only if there is a positively oriented frame $(\zeta_1,\zeta_2,\zeta_3)$ for $T\partial \widetilde Z_\lambda$ such that $\omega_f\wedge\omega_f (\widetilde V_\lambda,\zeta_1,\zeta_2,\zeta_3)>0$. But 
		$$\omega_f\wedge \omega_f(\widetilde V_\lambda, \zeta_1,\zeta_2,\zeta_3) = \iota_{\widetilde V_\lambda}\omega_f \wedge \omega_f (\zeta_1,\zeta_2,\zeta_3) = \widetilde \alpha_\lambda\wedge d\widetilde \alpha_\lambda(\zeta_1,\zeta_2,\zeta_3)$$
		so this condition is equivalent to asking that $\widetilde \alpha_\lambda \wedge d\widetilde \alpha_\lambda>0$.
				
		Therefore $(\widetilde Z_\lambda,\omega_f|_{\widetilde Z_\lambda}, \widetilde \eta_\lambda)$ is a Liouville filling of its boundary.
		
		In general a Weinstein structure is stronger than a Liouville structure: a Weinstein structure requires that the Liouville vector field is gradient-like for some Morse-function. However when the boundary is $\#^k(S^1\times S^2)$, the contact structure is planar, so for any Liouville filling, after possibly adding a trivial collar to the boundary, the symplectic structure is compatible with a subcritical Weinstein structure by \cite{Wendl} (relevant arguments on holomorphic disk filling of $\#^k(S^1\times S^2)$ can also be found in \cite{Eliash,CE}). Because the skeleton will lie on the interior of the domain (since the Liouville vector field is transverse to the boundary), the addition of the trivial collar does not make a difference for our purposes. Since there is a unique Weinstein structure on this filling by \cite{CE}, we can deform the Liouville vector field through Liouville vector fields for the fixed symplectic structure which are all positively transverse to the boundary until it agrees with the standard Weinstein structure.
		We conclude that $(\widetilde Z_\lambda, \omega_f|_{\widetilde Z_\lambda})$ supports a Weinstein domain structure as required.
	\end{proof}

	\begin{lemma}
		\label{l:symplectomorphic}
		Let $f:(X,\omega)\to (\cptwo,\omega_{FS})$ be a symplectic singular branched cover, where $X$ is a closed $4$-manifold. For $\omega_f=f^*\omega_{FS}+\varepsilon d\phi$ as above with $\varepsilon$ sufficiently small, $(X,\omega_f)$ and $(X,k\omega)$ are symplectomorphic, where $k$ is the degree of the branch cover.
	\end{lemma}
	
	\begin{proof}
		The cohomology classes of $\omega_f$ and $k\omega$ are the same because $[\omega_f] = [f^*\omega_{FS}+\varepsilon d\phi] =  [f^*\omega_{FS}] = k[\omega]$ where the last equality follows from Definition~\ref{def:ssbc}. Also by definition of a symplectic singular branched cover, the $2$--forms 
		$$tf^*\omega_0 +(1-t)k\omega$$ 
		are symplectic for all $t\in [0,1)$. Therefore for $\varepsilon$ sufficiently small, the $2$--forms 
		$$\omega_t:=t(f^*\omega_0+\varepsilon d\phi)+(1-t)k\omega$$
		are still symplectic for all $t\in [0,1)$. For $t=1$, $\omega_1=\omega_f$. Therefore we have a $1$--parameter family of symplectic forms $\omega_t$ in the same cohomology class interpolating between $\omega_f$ and $k\omega$, so by Moser's theorem\cite{Moser}, there is an ambient isotopy $\Psi_t\colon X\to X$ such that $\Psi_t^*(\omega_t)=\omega_0 =k\omega$. Therefore $\Psi_1$ gives the required symplectomorphism.
	\end{proof}
	
	We conclude this section with the proof of our main result.
	
	\begin{reptheorem}
		{thm:Main}
		Every closed, symplectic 4--manifold $(X,\omega)$ admits a Weinstein trisection.
	\end{reptheorem}
	
	\begin{proof}[Proof of Theorem~\ref{thm:Main}]
		Let $(X,\omega)$ be a closed, symplectic 4--manifold.  By Theorem~\ref{thm:AK}, there is a singular symplectic branched covering $f\colon X\to \CP^2$ such that the branch locus $Q$ is a symplectic braided surface.
		
		By Theorem~\ref{thm:symplbridgepos} and Lemma~\ref{l:sympliso}, we may assume after a symplectomorphism of $\cptwo$ that $Q$ is in bridge position with respect to a genus one trisection $\TT_0$ of $\CP^2$. 

		By Theorem~\ref{prop:branch-tri}, the preimages $W_\lambda = f^{-1}(Z_\lambda)$ of the sectors $Z_\lambda$ of the trisection $\TT_0$ give a trisection $\widetilde\TT = f^*(\TT_0)$ of $X$.
		
		By Theorem \ref{thm:pullbackWeiTri}, $X$ carries a symplectic structure $\omega_f$ such that the trisection $\TT=f^*\TT_0$ is Weinstein with respect to $\omega_f$. By Lemma \ref{l:symplectomorphic}, there is a symplectomorphism $\Psi\colon (X,\omega_f) \to (X,k\omega)$, where $k$ is the degree of the covering. Therefore $\Psi_*(\TT)$ gives a Weinstein trisection on $(X,k\omega)$. 
		
		Let $W_\lambda$ be the three sectors of $\Psi_*(\TT)$ on $X$ for $\lambda=1,2,3$ and let $\omega_\lambda = \omega|_{W_\lambda}$. Then there are Weinstein strutures $(W_\lambda, k\omega_\lambda, V_\lambda, \phi_\lambda)$. The Weinstein conditions on $V_\lambda$ and $\phi_\lambda$ are that $d\iota_{V_\lambda}(k\omega_\lambda) = k\omega_\lambda$ and $V_\lambda$ is gradient-like for $\phi_\lambda$. Because $k$ is a constant, this implies $d\iota_{V_\lambda}(\omega_\lambda) = \omega_\lambda$. In other words $V_\lambda$ is also a Liouville vector field for $\omega_\lambda$. Therefore globally rescaling the symplectic form by a constant, we still have Weinstein structures on the three sectors $(W_\lambda, \omega_\lambda, V_\lambda, \phi_\lambda)$. Therefore $\Psi_*(\TT)$ is a Weinstein trisection on $(X,\omega)$.
	\end{proof}
	
	\section{Examples}
	\label{sec:Examples}
	
	In this section we construct a number of examples of Weinstein trisections via the following procedure.
	First, we find a (singular) bridge trisected surface in $\cptwo$, we appeal to techniques of the first author~\cite{LC-Symp} to argue that the surface is a symplectic braided surface.  Next, we describe a branched covering $f$ of the surface and use Theorem~\ref{thm:pullbackWeiTri} to conclude that the pull-back trisection is Weinstein with respect to the (perturbed) lift $\omega_f$ of $\omega_{FS}$.
	
	\subsection{Weinstein trisections for complex hypersurfaces in $\CP^3$}
		\label{ex:hyper}
		Let $S_d$ denote a smooth degree $d$ complex hypersurface in $\CP^3$.  Let $\TT_{S_d}$ denote the efficient trisection of $S_d$ constructed in~\cite[Section~4]{LM-Complex} as the $d$--fold cyclic covering of the genus one trisection of $\CP^2$, branched along an efficient bridge trisection of a curve $\cC_d$ of degree $d$ in $\CP^2$. A priori, the efficient bridge trisection given in~\cite{LM-Complex} corresponds to a surface $\cS$ that is smoothly isotopic to $\cC_d$. However, by~\cite[Proposition~3.3]{LC-Symp}, $\cS$ is isotopic through bridge trisected surfaces to a symplectic surface. Now, by Theorem~\ref{thm:pullbackWeiTri}, we have that the trisections $\TT_{S_d}$ are Weinstein.
	
	\subsection{A Weinstein trisection for $\CP^2\#\overline\CP^2$}
		\label{ex:TCQ}
		Consider the surface $\cC$ described by the shadow diagram in Figure~\ref{fig:lifts}(a). (See~\cite{LM-Complex} for details regarding shadow diagrams.)  This diagram corresponds to a singular bridge trisection, with each patch given as the cone on a right-handed trefoil. We encourage the reader to verify that each of the links
		$$(S^3_\lambda,K_\lambda) = (H_\lambda,\cT_\lambda)\bigcup_{(\Sigma,\bold x)}\overline{(H_{\lambda+1},\cT_{\lambda+1})}$$
		is a right-handed trefoil: For each $\lambda\in\Z_3$, find an orientation-preserving diffeomorphism from the shadow diagram onto the genus one Heegaard surface in $S^3$ such that $H_\lambda$ is inside the torus and $\overline H_{\lambda+1}$ is outside, and consider the knot obtained by perturbing the interiors of the shadow arcs into their respective solid tori.
		
		To see that this surface has degree four, we briefly recall how the homology of $\cptwo$ can be understood from its trisection diagram (cf.~\cite{FKSZ}). Following~\cite[Subsection~6.2]{LM-Complex}, we note that
		$$H_2(X;\Z) = L_\gamma\cap(L_\alpha+L_\beta),$$
		where $L_\lambda\subset H_1(\Sigma;\Z)$ is the subspace spanned by the curve $\lambda$.  Note that $[\gamma] = -[\alpha]-[\beta]$ in $H_1(\Sigma;\Z)$. It follows that $[\cpone] = [\gamma] = -[\alpha]-[\beta]$. Let $\phi$ denote the triangular region of $\Sigma\setminus(\alpha\cup\beta\cup\gamma)$ containing the positive bridge point, together with a compressing disk bounded by each of the $\lambda$ in $H_\lambda$.  Then, $\phi$ is a 2--chain in $\cptwo$ representing $[\cpone]$ in $H_2(\cptwo;\Z)$.
		
		To apply Proposition~6.1 of~\cite{LM-Complex}, we need a shadow diagram for $\cC$ where the shadow arcs are disjoint from the cut curves for their respective solid tori. This could be accomplished by isotoping the cut curves (cf. Figure~\ref{fig:cusp}(c)), but then we would need to recalculate $\phi$.  Instead, we perturb the bridge trisection to be 4--bridge, as shown in Figure~\ref{fig:cusp}(b).  Now, it follows that $[\cpone]\cdot[\cC] = 4$, since there are now four positive bridge points in the triangular region of $\phi$.
		
		Using the techniques of Section~3 of~\cite{LC-Symp}, combined with a slight generalization of~\cite[Propositions~3.2 and~3.3]{LC-Symp} to a setting that includes cusp singularities, it can be shown that $\cC$ is isotopic through singular  bridge trisections to a symplectic surface. Thus, $\cC$ is a symplectic tri-cuspidal quartic. (Note that by~\cite{GollaStarkston} there is a unique symplectic isotopy class of tri-cuspidal quartics.)
		In what follows, we will describe a singular, irregular, 3--fold branched covering
		$$f\colon \cptwo\#\overline\CP^2\to\cptwo$$
		that is branched along the tri-cuspidal quartic $\cS$. Applying Theorem~\ref{prop:branch-tri}, we will lift a genus one Weinstein trisection of $\cptwo$ to a genus two trisection of $\cptwo\#\overline\CP^2$, and this lifted trisection will be Weinstein by Theorem~\ref{thm:pullbackWeiTri}.
		
		Figure~\ref{fig:cusp}(c) is a doubly pointed trisection diagram for $(\CP^2,\cC)$ that is obtained from Figure~\ref{fig:cusp}(a) by isotoping the $\lambda$--curves to be disjoint from the shadow arcs and forgetting the shadow arcs.  (The arcs can be recovered uniquely by connecting the two bridge points in the complement of the respective $\lambda$--curve; see~\cite{GayMei} for more discussion of doubly pointed trisection diagrams and the connection to knotted spheres in 4--manifolds and note the relevant connection here that $\cC$ is rational.) Figure~\ref{fig:cusp}(d) is an isotopic doubly pointed trisection diagram that has the feature that the $\alpha$, is now parallel to an edge of the square representing the torus.

	\begin{figure}[h!]
		\centering
		\includegraphics[width=.76\textwidth]{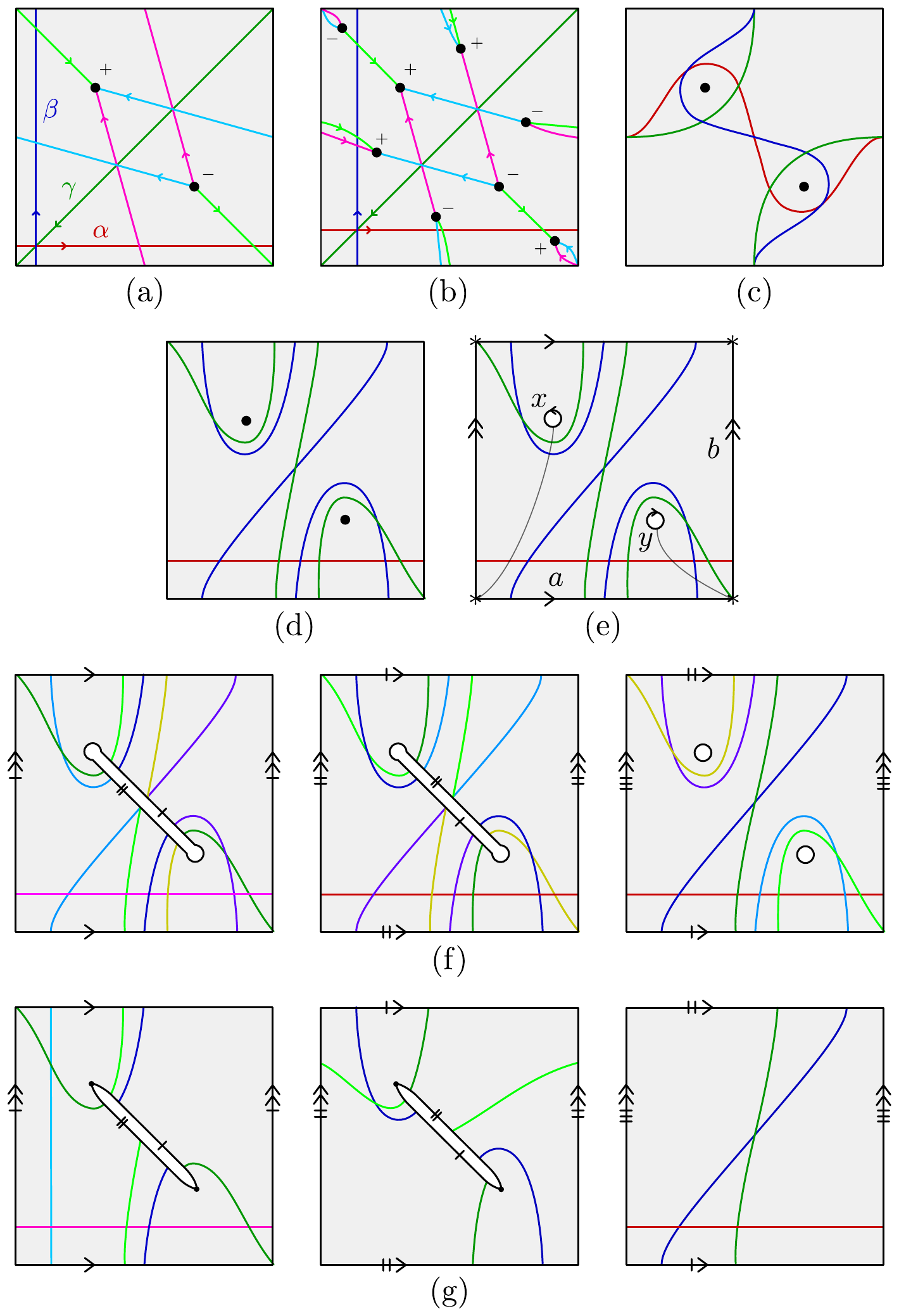}
		\caption{(a)--(d) Diagrams for a tri-cuspidal quartic. (e) The diagram from (c) with neighborhoods of the bridge points removed. (f) The three sheets corresponding the irregular 3--fold cover of the punctured surface. (g) A trisection diagram (after identification) for the irregular 3--fold branched cover of the tri-cuspidal quartic.}
		\label{fig:lifts}
	\end{figure}
		
		Figure~\ref{fig:cusp}(e), which is just this new doubly pointed trisection diagram with disk neighborhoods of the bridge points removed, can be used to calculate $\pi_1(\CP^2\setminus\nu(\cC))$ as follows.  Choose as base-point the corner(s) of the square, let the classes of the horizontal and vertical edges be $a$ and $b$, respectively.  Connect the punctures to the corners with whiskers as shown, and denote the loops following these arcs and encircling the punctures $x$ and $y$, respectively.
		Then, the compressing disk bounded by $\alpha$ gives the relation $a=1$, the compressing disk bounded by $\beta$ gives the relation $y^{-1}a^{-1}b^{-1}x^{-1}bayb = 1$, and the compressing disk bounded by $\gamma$ gives the relation $y^{-1}b^{-1}y^{-1}a^{-1}b^{-1}x^{-1}b = 1$.
		We also get the relation $ayba^{-1}b^{-1}x^{-1}=1$ from the core surface. From all this, we get the presentation
		$$\pi_1(\CP^2\setminus\nu(\cC)) = \langle b,y\,|\, yby=byb, y^2=b^{-2}\rangle,$$
		noting that $x=y$.  The fact that this choice of base-point and whiskers yields $x=y$ and $a=1$ will make drawing the branched cover easier, since this presentation gives rise to a particularly nice representation to $\mathscr S_3$.
		
		\begin{remark}
			We remark in passing that this group is closely related to several well-known groups via the following sequence of quotients:
			$$\mathcal B_3 \cong \pi(T_{3,2}) \longtwoheadrightarrow \pi(\cC)\longtwoheadrightarrow\pi(\cP_{T_{3,2}})\cong \text{PSL}_2(\Z).$$
			Here, $\mathcal B_3$ is the braid group on three strands, $\pi(T_{3,2})$ is the group of the right-handed trefoil, $\pi(\cC)$ is the group of the tri-cuspidal quartic, $\pi(\cP_{T_{3,2}})$ is the group of the knotted projective plane in $S^4$ obtained by forming the connected sum of the spun trefoil with the standard cross-cap, and $\text{PSL}_2(\Z)$ is the modular group.
		\end{remark}
		
		Consider the representation $\rho\colon\pi_1(\CP^2\setminus\nu(\cC))\twoheadrightarrow \mathscr S_3$, where $y\mapsto (1\,2)$ and $b\mapsto (2\,3)$. Note that we must have $a\mapsto 1$ and $x\mapsto (1\,2)$.  We remark that this representation is the unique (up to conjugation) representation onto $\mathscr S_3$, since any such representation is determined by choosing distinct transpositions for $\rho(y)$ and $\rho(b)$.  Let $f^\circ\colon X^\circ\to\CP^2\setminus\nu(\cC)$ be the irregular 3--fold (unbranched) covering corresponding to $\rho$. This covering can be extended to a (non-singular) branched covering over $\cptwo\setminus\nu(\text{the cusps of $\cC$})$.  On the boundary of the neighborhood of the cusps, this branched covering induces the unique irregular 3--fold covering of $S^3$ over its self with branch set the right-handed trefoil, as in Examples~\ref{ex:cone-br}(2).  Thus, we can cone off over the cusps to obtain a singular branched covering $f\colon X\to \cptwo$ with only cusp singularities.  In particular, $f$ satisfies the hypotheses of Theorem~\ref{thm:pullbackWeiTri}.  Note that $f$ is the unique such map, by our analysis of $\rho$.
		
		To describe a trisection of the cover $X$, we first describe $\widetilde\Sigma = f^{-1}(\Sigma)$.  Let $\Sigma^\circ$ denote the core surface $\Sigma$, cut along an arc connecting the punctures and along a curve parallel to $\alpha$. Then $\Sigma^\circ$ is a punctured annulus whose core curve is given by $a$. Let $\widetilde\Sigma^\circ$ denote three copies of $\Sigma^\circ$, and note that the restriction $f\vert_{\text{Int}(\widetilde\Sigma^\circ)}$ can be thought of as the three-to-one map $\text{Int}(\widetilde\Sigma^\circ)\to\text{Int}(\Sigma^\circ)$, since $\rho(a)=1$.

		 The three sheets, $\widetilde\Sigma^\circ_1$, $\widetilde\Sigma^\circ_2$, and $\widetilde\Sigma^\circ_3$, of $\widetilde\Sigma^\circ$ are shown in Figure~\ref{fig:lifts}(f).  To recover $\widetilde\Sigma$, we must understand the identifications of these sheets described in this figure.  In the left sheet, the sides of $\widetilde\Sigma^\circ_1$ corresponding to $a$ are re-glued as they are in $\Sigma$; since $\rho(b)=(2\,3)$, the curve $b$ lifts to a curve in this first sheet.  Now, the first sheet is a punctured torus with boundary coming from the cut along the arc connecting the punctures.  This puncture is identified, as depicted, with the corresponding puncture in the second sheet.  This is because this puncture came from cutting open $\Sigma^\circ$ along an arc that connected the punctures $x$ and $y$, each of whose image under $\rho$ is $(1\,2)$.  For the same reason, in the third sheet, the cut along this arc is re-glued; the curves $x$ and $y$ lift to curves in the third sheet.  Finally, the two $a$--curves of the third sheet are identified with the two $a$--curves of the second sheet, as shown. The relevant fact here is that $\rho(b)=(2\,3)$, so $b$ is double covered by a curve contained in the union of the second and third sheets, which is a punctured torus. Finally, the meridional punctures $x$ and $y$ are filled in in the third sheet; this gives a branched covering of surfaces.
		 
		 By Theorem~\ref{prop:branch-tri}, the genus one trisection of $\cptwo$ will lift to a trisection of $\CP^2\#\overline{\CP^2}$.  However, we will now check explicitly that the cover of the core surface described above extends over the other pieces of the trisection.
		 It remains to check that the map $f\colon\widetilde\Sigma\to\Sigma$ described above can be extended to three branched coverings of handlebodies.  Since $f^{-1}(\lambda)$ is a collection of curves for each $\lambda$, it suffices to check that each such collection is a cut system for $\widetilde\Sigma$.  We saw above that $\widetilde\Sigma$ is the union of two punctured tori, hence a genus two surface.  So, we need that $f^{-1}(\lambda)$ contains two non-parallel, non-separating curves for each $\lambda$.  This is verifiable in Figure~\ref{fig:lifts}(f); each $f^{-1}(\lambda)$ consists of three curves, two of which are parallel.  In Figure~\ref{fig:lifts}(g), a redundant $\lambda$--curve has been deleted, and isotopy of the remaining two curves has been performed.
		 		 
		 The result is that Figure~\ref{fig:lifts}(f) is a genus two trisection diagram for $\CP^2\#\overline\CP^2$.  To verify this last claim, one could appeal to~\cite{MZ-Genus} to deduce that the diagram is either reducible or corresponds to $S^2\times S^2$.  Then, one could calculate homology following~\cite{FKSZ} or~\cite[Subsection~6.2]{LM-Complex}.  Alternatively, one could perform handleslides to find a reducing curve (a separating curve disjoint from all six cut curves) and verify explicitly that the two resulting summands correspond to $\cptwo$ and $\overline\CP^2$.

	\bibliographystyle{amsalpha}
	\bibliography{WeiTri-References}


\end{document}